\documentclass[11pt,reqno]{amsart}

 \usepackage{amsmath,amsthm,amsfonts}
\usepackage{latexsym,txfonts}
\usepackage{amssymb}
\usepackage{amsrefs}

\newcommand{\wt}[1]{\widetilde {#1}}
\newcommand{\ti}{\tilde}
\newcommand{\R}{\mathbb{R}}

\newcommand{\Sf}{\mathbb{S}}

\renewcommand{\bar}{\overline}

\numberwithin{equation}{section}

\newtheorem{thm}{Theorem}[section]
\newtheorem{cor}[thm]{Corollary}
\newtheorem{lem}[thm]{Lemma}
\newtheorem{prop}[thm]{Proposition}

\theoremstyle{remark}

\newcommand{\Del}[1]{}

\def\sign{\mathrm{sign}}

\begin{document}

\author{Elisabetta Chiodaroli and Joachim Krieger}
\title{A class of large global solutions\\ for the Wave--Map equation}
\maketitle
\centerline{EPFL Lausanne}

\centerline{Station 8, CH-1015 Lausanne, Switzerland}

\begin{abstract}
In this paper we consider the equation for equivariant wave maps from $\R^{3+1}$ to $\Sf^3$ 
and we prove global in forward time existence of certain $C^\infty$-smooth 
solutions which have infinite critical Sobolev norm $\dot{H}^{\frac{3}{2}}(\R^3)\times \dot{H}^{\frac{1}{2}}(\R^3)$.
Our construction provides solutions which can moreover satisfy the additional size condition $\|u(0, \cdot)\|_{L^\infty(|x|\geq 1)}>M$ for arbitrarily chosen $M>0$.
These solutions are also stable under suitable perturbations.
Our method, strongly inspired by \cite{krsc}, is based on a perturbative approach around suitably constructed approximate 
self--similar solutions.

\end{abstract}
\section{Introduction}
\subsection{Corotational wave maps}
Let $M$ be the Minkowski space $\R^{n,1}$, with coordinates $x=(x^0, x^1,\dots, x^n)=(t, \bar{x})$ and let $N$
be a smooth, complete, rotationally symmetric $k$-dimensional Riemannian manifold without boundary.
Following Tachikawa \cite{tac}, we can then identify $N$, as a warped product, with a ball of radius $R\in \R^+ \cup \{\infty\}$
in $\R^k$ equipped with a metric of the form
\begin{equation}
 ds^2=du^2+g^2(u) d\theta^2,
\end{equation}
where $(u, \theta)$ are polar coordinates on $\R^k$, $d\theta^2$ is the standard metric on the sphere $\Sf^{k-1}$,
and $g: \R \rightarrow \R$ is smooth and odd, 
\begin{equation}\label{eq:g}
 g(0)=0, \quad g'(0)=1.
\end{equation}
With these notations in hand, we can define a \textit{wave map} $U:M\rightarrow N$ as a stationary point (with respect
to compactly supported variations) of the functional
\begin{equation}\label{eq:functional}
 \mathcal{L}[U]= \frac{1}{2} \int_M \langle \partial_{\mu} U, \partial^{\mu} U \rangle
= \frac{1}{2}\int_M \partial_{\mu} u \partial^{\mu} u +g^2(u) \Gamma_{A,B} \partial_{\mu} \theta^A \partial^\mu \theta^B.
\end{equation}
Hence, they satisfy the Euler-Lagrange equations: if we denote the vector valued map $U$ as $U:=(u, \theta) \in \R^k$ then 
it satisfies
\begin{equation}\label{eq:Euler-Lagrange}
\bigg \{
\begin{array}{r}
\partial_{\mu} \partial^{\mu} u +g(u)g'(u) \Gamma_{A,B} \partial_{\mu} \theta^A \partial^{\mu} \theta^B \; =\;0 \\
\partial^{\mu} (g^2(u)\Gamma_{A,B} \partial_{\mu} \theta^B)\;=\;0. \\
\end{array}
\end{equation}
We also introduce spatial polar coordinates $(t, r, \omega)\in \R\times \R^+\times \Sf^{n-1}$ on $M$.
In these coordinates the metric on $M$ takes the form
$$ds^2= - dt^2+dr^2+r^2 d\omega^2.$$
Since $N$ is assumed to be rotationally symmetric, it becomes natural to consider \textit{equivariant}
wave maps by requiring that the orbit of any point in $M$ under spatial rotations maps into the orbit of
the image point in $N$. When this mapping of the orbits has degree $1$, we call the map \textit{corotational}.
We suppose to be in the equivariant framework, thus we require that
\begin{equation}\label{eq:corotational}
 u= u(t,r) \quad \text{and} \quad \theta=\theta(\omega).
\end{equation}
It then follows that $\theta: \Sf^{n-1} \rightarrow \Sf^{k-1}$ has to be an \textit{eigenmap}, i.e. a 
harmonic map of constant energy density
$$e=|\nabla_{\omega} \theta|^2.$$
Under this hypothesis, the wave map system \eqref{eq:Euler-Lagrange} simplifies
and reduces to the following simple scalar wave equation for the
spatially radial function $u:M \rightarrow \R$
\begin{equation}\label{eq:scalar wave equation general}
 u_{tt} - u_{rr} - \frac{n-1}{r}u_r + \frac{f(u)}{r^2} = 0,\quad f(u) = eg(u) g'(u)
\end{equation}

\subsection{Wave maps into $\Sf^3$}
We are interested in equivariant corotational wave maps from $M=\R^{3+1} \rightarrow \Sf^3=N$.
In this case, it is immediate to represent the target manifold $\Sf^3$, as a warped product.
Let us show the standard way how the standard $3$-dimensional sphere $\Sf^3$ can be written with a rotationally
symmetric metric of the type $du^2 + g^2(u) ds_2^2$ where $ds_2^2$ is 
the canonical metric on $\Sf^2\subset \R^3$.
We consider the map
\begin{align*}
 I  :  (0, \pi) \times \Sf^2 &\rightarrow \R\times \R^3 \\
 I(u, \theta) &= (\cos u, \sin u \cdot \theta)
\end{align*}
which maps into the unit sphere in $\R^4$.
In order to see that $I$ is a Riemannian isometry we compute the canonical metric on $\R\times \R^3$ (``can'')
using the coordinates $(\cos u, \sin u \cdot \theta)$.
In order to  carry out the computation we use that 
\begin{align*}
 1 &= (\theta^1)^2+ (\theta^2)^2+ (\theta^3)^2 \\
 0 &= 2( \theta^1 d\theta^1+ \theta^2 d\theta^2 + \theta^3 d\theta^3).
 \end{align*}
Thus we obtain
\begin{align*}
 \text{can}&= (d\cos u )^2+ \sum \delta_{i,j} \; d(\sin u \; \theta^i)\; d(\sin u\; \theta^j)  \\
 &= \sin^2 u \; du^2 + \sum \delta_{ij} \left( \theta^i \cos u \; du + \sin u \; d\theta^i\right) \left(\theta^j \cos u\; du + \sin u \;d\theta^j \right)\\
 &= \sin^2 u \; du^2 +\cos^2 u \; du^2 + \sin^2 u \left( \sum (d \theta^i)^2\right) \\
 &= du^2 + \sin^2 u \left( \sum (d \theta^i)^2 \right),
\end{align*}
and the claim follows from the fact that $\sum (d \theta^i)^2$ is exactly the canonical metric $ds_2^2$.
According to our notation, the function $g$ is here chosen to be $g(u) := \sin (u)$ and it satisfies $g(0)=0$ and $g'(0)=1$.
Since we are dealing with $M= \R^{3+1}$, we have here $n=k=3$: for  $(t, r, \omega)\in \R\times \R^+\times \Sf^{2}$ on $M$ and $U:=(u, \theta)\in \R\times \Sf^2$
we can choose the map $\theta= \theta(\omega)$ to be the identity map from $\Sf^2$ to $\Sf^2$, whence $e=2$. The same ``ansatz'' was also considered in \cite{scu} for harmonic maps into spheres.
It is now clear that the problem of looking for equivariant and corotational wave maps from $\R^{3+1}$ to $\Sf^3$
reduces to solving the following equation (see \eqref{eq:scalar wave equation general})
\begin{equation}\label{eq:scalar wave equation}
 u_{tt} - u_{rr} - \frac{2}{r}u_r + \frac{f(u)}{r^2} = 0, \quad f(u) = 2\sin u \;\cos u
\end{equation}
on $\R^{1+1}$.

\section{Background literature and main result}

The initial value problem for the
wave map equation has drawn a lot of attention in the mathematical community in the last twenty years. In particular, 
wave maps have been studied extensively in the case of a flat background $M$. It is impossible here to account for
the large amount of publications in the subject, hence we will quote some important contributions.

Let us make a short digression to explain the main problems related to the wave maps equation with particular
attention to the case of flat backgrounds.
As for other nonlinear evolutionary equations the first issues one is confronted with
are the existence of global classical solutions and the development of singularities.
The local--in--time existence is rather standard but it is challenging to identify classes of initial data for which global existence holds or on the contrary 
the identification of specific initial data that lead to a breakdown (blow--up) of the solution in finite time.
If blow--up occurs it is interesting to get the idea behind its formation.
For scaling invariant equations with a positive energy, such as the wave maps equation (see \cite{dosca}), heuristically one would expect
finite--time blow--up when the shrinking of the solution is energetically favorable, i.e. in the so called energy supercritical case,
while global existence should occur when shrinking to smaller scales is energetically prohibited, i.e. in the energy subcritical space.
The limit case when the energy itself is scaling invariant is called energy critical.
For wave maps, the criticality classification depends on the spatial dimension of the base manifold $M$. The equation is
energy subcritical, critical, or supercritical if dim$M=1+1$, dim$M=2+1$ or dim$M\geq3+1$, respectively. 

As anticipated in the previous section, in this paper we study the simplest energy supercritical case: corotational wave
maps from $(3+1)$- Minkowski space to the three sphere which satisfy equation \eqref{eq:scalar wave equation}.
Global well-posedness of the Cauchy problem for this equation when data are small in a sufficiently high Sobolev space follows
from \cite{sid}. 
The existence of global weak solutions is shown in \cite{sh} for any initial data of finite energy, but otherwise arbitrary.
Furthermore, a number of results concerning the Cauchy problem for
equivariant wave maps are obtained in \cite{sta}; in particular, local well-posedness with minimal regularity requirements for the initial data is studied.  
On the other hand, in \cite{sh}, Shatah has shown that in the case $M=\R^{3+1}$ and $N= \Sf^3$ the corotational wave map problem \eqref{eq:scalar wave equation}
admits self--similar solutions, thus making it possible 
to pose Cauchy problems with smooth data whose solutions develop singularities in finite time. This 
initiated the construction of blow--up solutions in the form of self-similar solutions: for further generalizations
to more general targets and examples of non--uniqueness using self--similar blow-up profiles see \cite{sta} and \cite{casta}.
As usual, by exploiting finite speed of propagation,a self-similar solution can be used to construct a solution with compactly supported
initial data that breaks down in finite time. In fact, \eqref{eq:scalar wave equation} admits many self-similar solutions \cite{biz} and a particular one
exhibiting blow--up was given in closed form in \cite{tusp} and is known as \textit{ground state} or \textit{fundamental self--similar solution}. In \cite{don}, the blow--up of the ground state
is shown to be stable indicating that blow--up is a generic situation, see also \cite{donac}, \cite{donac1} and \cite{dosca}.
It is clear that self--similar solutions play a crucial role in the blow--up theory for equation \eqref{eq:scalar wave equation}; they correspond to self--similar 
data at the time of blow--up. Of course these solutions leave the standard scaling critical Sobolev space $\dot{H}^{\frac{3}{2}} $ at the blow--up time.
In \cite{ge}, the author introduces a Besov space--based framework which includes the blowing up solutions of Shatah \cite{sh}
and Bizo{\'n} \cite{biz}; in particular two types of solutions of \eqref{eq:scalar wave equation} 
are constructed: on the one hand existence, uniqueness and scattering of solutions starting from data which possess infinite critical Sobolev norm,
but are small in the sense of suitable Besov spaces is proven (i.e. in these Besov spaces blow--up does not occur), on the other hand 
large (in the sense of Besov spaces) data can be considered only in the strictly self--similar case where existence still holds but uniqueness is lost.
\smallskip

Moving from this background literature, where essentially global well-posedness for equation \eqref{eq:scalar wave equation} is known
only for finite critical norm and self--similar blow--up seems a generic situation, we present a completely new result on global existence of smooth 
solutions of \eqref{eq:scalar wave equation} with infinite critical Sobolev norm $\dot{H}^{\frac{3}{2}}\times \dot{H}^{\frac{1}{2}}$.
Moreover, we can construct such solutions so that they satisfy also the size condition $\|u(0, \cdot)\|_{L^\infty(r\geq 1)}>M$ for arbitrarily chosen $M>0$.
The building block for our construction are self--similar solutions which are smooth away from the light cone but singular across it.
Our method, strongly inspired by \cite{krsc}, first consists in regularizing the self--similar solutions near the light cone, thus obtaining
approximate self--similar solutions and then proceeds by solving a perturbative problem around the approximate self--similar solutions so to generate
global solutions. In fact, the regularization destroys the scaling invariance and this turns out
to be important for the ensuing perturbative argument. 
The smooth data thus constructed have infinite critical norm 
$$\|u[0]\|_{\dot{H}^{\frac{3}{2}}(\R^3)\times \dot{H}^{\frac{1}{2}}(\R^3)}=\infty$$
only because of insufficient decay at
infinity and not because of some singular behavior in finite space-time.

More precisely the main theorem of this paper is the following.

\begin{thm}\label{th:main}
The equation \eqref{eq:scalar wave equation} for corotational equivariant wave maps from $\R^{3+1}$ to $\Sf^3$ admits
 smooth data $(f,g)\in C^{\infty} \times C^{\infty}$ 
decaying at infinity to zero, satisfying
$$\| (f,g)\|_{\dot{H}^{\frac{3}{2}}(\R^3)\times \dot{H}^{\frac{1}{2}}(\R^3)}=\infty \qquad \text{but}\qquad \| (f,g)\|_{\dot{H}^{s}(\R^3)\times \dot{H}^{s-1}(\R^3)}<\infty $$
for any $s>\frac{3}{2}$ and such that the corresponding evolution of \eqref{eq:scalar wave equation} exists globally in forward time as 
a $C^\infty$--smooth solution. 
In fact the initial data can be chosen such that 
$$\| f\|_{L^\infty(r\geq1)}> M,$$
for arbitrary $M>0$.
The solutions thus obtained are stable with respect to a certain class of perturbations.
\end{thm}

Moreover, the solutions of Theorems \ref{th:main} exhibit a precise asymptotic description.
Let us remark that the possibility of achieving  
$\| f\|_{L^\infty(r\geq1)}> M$ leads to solutions which are not even small in Besov spaces (such as $\dot{B}^{\frac{3}{2}}_{2,\infty} \times \dot{B}^{\frac{1}{2}}_{2,\infty}$)
as it is instead the case for the global solutions constructed in \cite{ge}.
\smallskip

The paper is organized as follows.
In Section \ref{sec:self-similar} we construct smooth self-similar solutions of the form
$$Q_0(t,r)= Q\left(\frac{r}{t}\right)$$
for $r>t$ or $r<t$ by a reduction to a nonlinear Sturm-Liouville problem, see \eqref{eq:self-similar eq}. We solve this
ODE by use of fixed point arguments using smallness in $L^\infty$.
As we will see, starting with small data at $a= 0$, $Q_0$ exhibits a singularity of
the form $|a-1| \log |a-1|$ near $a=1$ which precisely fails logarithmically to belong to the
scaling critical Sobolev space $\dot{H}^{\frac{3}{2}}(\R^3)$. 
In the second part of  Section \ref{sec:self-similar}
we glue together the two solutions residing
inside and outside the light-cone, respectively, at $r=t$ to form a continuous
function $Q_0(r,t)$ which decays at $a=\infty$ as $a^{-1}$
(and thus fails to lie in $\dot{H}^{\frac{3}{2}}$ at $\infty$).

Section \ref{sec:large} shows that one of the
parameters determining the self--similar solution of Section \ref{sec:self-similar} near the singularity at $a=1$ can be chosen
arbitrarily large, leading to rapid growth and oscillation of the solution on the set
$a>1$ near $a=1$; the nice behavior of non--linear terms allows to extend these
solutions all the way to $a=\infty$, where they again decay asymptotically like $a^{-1}$. Thanks to this extension
we can achieve the large solutions announced in Theorem \ref{th:main}.

Section \ref{sec:regularization} is devoted to the regularization of $Q_0$ near the light cone: we multiply the singular components of $Q_0$
by a smooth cut-off function localized at $a=1$ which leaves untouched the solution far from the light cone. The smooth function thus obtained
is no longer an exact solution of \eqref{eq:scalar wave equation}, but in the next section
 we show that it can be perturbed in a smooth way to obtain exact solutions. 
The concluding argument is the content of Section \ref{sec:completion}: it relies on the energy
supercritical nature of the problem and imitates the ideas already present in \cite{krsc}.
The methods of Sections \ref{sec:regularization} and \ref{sec:completion} apply to the case of small self--similar solutions constructed in Section \ref{sec:self-similar} as well as 
to large--size self--similar solutions as constructed in Section \ref{sec:large}.

\section{Self--similar solutions} \label{sec:self-similar}
A \textit{self--similar} solution of the Cauchy problem for a wave map $u$ is a solution that depends only on the ratio $r/t$.
Such solutions are thus constant along rays emanating from the origin in space--time, and consequently experience a gradient singularity
at the origin (if non--trivial).
The existence of non--trivial self-similar solutions was proven in the case of $M= \R^{3+1}$ (which we are concerned with) first by Shatah in \cite{sh} where $N=\Sf^3$, then it was extended to more general rotationally symmetric, non--convex targets by Shatah and Tahvildar-Zadeh in \cite{sta}
and also extended to higher dimensions for $M$ by the same authors in a joined work with Cazenave \cite{casta}; a particular self--similar solution exhibiting 
blow--up was explicitly given in \cite{tusp},

In the corotational setting, equation \eqref{eq:scalar wave equation} admits a self--similar solution 
\[
u(t, r) = Q\left(\frac{r}{t}\right)
\]
if $Q$ satisfies the following ordinary differential equation in $a:= r/t$
\begin{equation}\label{eq:self-similar eq}
 (1-a^2)Q''(a) + \left( \frac{2}{a} - 2a \right)Q'(a)-\frac{1}{a^2}f(Q(a)) = 0,
\end{equation}
with $f$ defined as above to be $f(Q) = 2 \sin Q  \cos Q$. 
Indeed we can compute
\begin{align*}
u_{t} &= -\frac{r}{t^2}Q'\left(\frac{r}{t}\right),\quad u_{tt} = \frac{2r}{t^3}Q'\left(\frac{r}{t}\right) + \frac{r^2}{t^4}Q''\left(\frac{r}{t}\right) \\
u_{r} &= \frac{1}{t}Q'\left(\frac{r}{t}\right),\quad u_{rr} = \frac{1}{t^2}Q''\left(\frac{r}{t}\right)
\end{align*}
whence equation \eqref{eq:self-similar eq}.
The natural initial conditions at $a=0$ are
\begin{equation}\label{eq:initial conditions}
 Q(0)=0, \quad Q'(0)=d_0.
\end{equation}

We immediately observe that the possible singularities for a solution of \eqref{eq:self-similar eq} on $[0,1]$
can occur only at $a=0$ and $a=1$.
In the following we will show existence of exact solutions to \eqref{eq:self-similar eq} by carefully analysing the behavior
near $a=0$ and $a=1$. The strategies involved here are inspired by the analogous construction in \cite{krsc}.

\begin{lem}\label{l:lemma1}
There exists $\varepsilon>0$ small enough such that, for any $0\leq d_0 \leq \varepsilon$, the equation \eqref{eq:self-similar eq}
admits a unique smooth solution on $[0, 1/2]$ with initial conditions \eqref{eq:initial conditions}.
Furthermore
\begin{align}\label{eq:solution in 1/2}
 Q(1/2)&= d_0 \varphi_0 (1/2) + O(d_0^3) \notag \\
  Q'(1/2)&= d_0 \varphi_0'(1/2) + O(d_0^3),
\end{align}
where 
$$\varphi_0 (a)=\frac{3}{4}\left[\frac{2}{a}+ \frac{1- a^2}{a^2}\log \left(\frac{1-a}{1+a}\right)\right].$$
\end{lem}

\begin{proof}
We consider the associated linear equation
\begin{equation}\label{eq:linear1}
  Q''(a) + \frac{2}{a} Q'(a)-\frac{2}{a^2(1-a^2)}Q(a) = 0,
\end{equation}
with fundamental system
\begin{equation}\label{eq:fundamental system}
\varphi_1(a)= \frac{1-a^2}{a^2}, \qquad \varphi_2(a)= \frac{2}{a}+ \frac{1- a^2}{a^2}\log \left(\frac{1-a}{1+a}\right).
\end{equation}
We define the Green function $G(a,b)$ for $0<b<a<1$
\begin{equation}\label{eq:Green function}
 G(a,b):= \frac{\varphi_1(a)\varphi_2(b)- \varphi_1(b)\varphi_2(a)}{W(b)}
\end{equation}
where the Wronskian $W$ is defined as
\begin{equation}\label{eq:Wronskian}
 W(b):=\varphi_1(b) \varphi_2'(b)- \varphi_1'(b) \varphi_2(b).
\end{equation}
By explicit computation we obtain

\begin{align}\label{eq:explicit Green function}
 W(b)&= \frac{4}{b^2}, \notag \\
 G(a,b)&=\frac{1-a^2}{4a^2}\left( 2b +(1-b^2) \log \left( \frac{1-b}{1+b}\right)\right)- \frac{1-b^2}{4} \left(\frac{2}{a} +\frac{1-a^2}{a^2} \log \left( \frac{1-a}{1+a} \right) \right). 
\end{align}

If we consider the inhomogeneous equation
\begin{equation}\label{eq:inhomogeneous eq}
  \psi''(a) + \frac{2}{a} \psi'(a)-\frac{2}{a^2(1-a^2)}\psi(a) = H(a)
\end{equation}
with initial conditions
\begin{equation}\label{eq:inhomogeneous initial conditions}
 \psi(0)=0, \quad \psi'(0)=0,
\end{equation}
then a solution in integral form can be obtained via the Green function as
$$ \psi(a)= \int_0^a G(a,b) H(b) db.$$
In view of this, we can make an ansatz for the self--similar solution of the non--linear equation \eqref{eq:self-similar eq} by
choosing $H$ as a non--linear function of $Q$ itself, i.e. by thinking of $H$
as the difference between the non--linear equation \eqref{eq:self-similar eq} and the linearized one \eqref{eq:linear1}:
more precisely we define $H$ as the following function of $Q(a)$ and $a$ 
\begin{equation}\label{F}
 H(Q(a)):= \frac{\sin (2Q(a)) - 2Q(a)}{a^2(1-a^2)}
\end{equation}
and we seek for a solution $Q$ of \eqref{eq:self-similar eq} on $0\leq a\leq 1/2$ with initial conditions \eqref{eq:initial conditions} of the form
\begin{equation}\label{eq:ansatz0}
 Q(a)= \frac{3}{4}d_0 \varphi_2(a) +\int_0^a G(a,b) H(Q(b))) db.
\end{equation}
We remark that $\varphi_2$ is analytic on $(-1,1)$ with expansion 
$$\varphi_2(a)= \frac{4}{3}a +\frac{4}{15} a^3+O(a^5),$$
whence $\varphi(0)=0$ and $\varphi'(0)=4/3$. For the sake of brevity we introduce $\varphi_0(a):=\frac{3}{4}\varphi_2(a)$ so
that we can simplify the ansatz \eqref{eq:ansatz0} to
\begin{equation}\label{eq:ansatz}
 Q(a)= d_0 \varphi_0(a) +\int_0^a G(a,b) H(Q(b)) db,
\end{equation}
so that the smallness parameter $d_0$ is highlighted.
In order to solve \eqref{eq:self-similar eq} on $[0,1/2]$ we set up a contraction argument by use of the ansatz \eqref{eq:ansatz}.
We define the map $T$ as
\begin{equation}\label{eq:contraction map}
 (Tf)(a):=d_0 \varphi_0(a) +\int_0^a G(a,b) H(f(b)) db,
\end{equation}
where $f$ belongs to some suitable functional normed space $X_{d_0}$ to be chosen such that $T:X_{d_0}\rightarrow X_{d_0}$ is a contraction.
According to our ansatz, we assume that $0\leq d_0\leq \varepsilon$ and we introduce the space $X_{d_0}$ 
$$X_{d_0}:= d_0 \varphi_0+ \left\{h(a)\; | \;h\in C^2([0,1/2]),\; \|h\|_{C^2[0,1/2]}\leq {d_0}^2,\; |h(a)|\leq {d_0}^2 a^2 \right\}.$$
The set $\left\{h(a)\; | \;h\in C^2([0,1/2]),\; \|h\|_{C^2[0,1/2]}\leq {d_0}^2,\; |h(a)|\leq {d_0}^2 a^2 \right\}$ defines a closed convex subset of
a linear space which we equip with the following norm
$$ \|h\|_{X_{d_0}}:= \|h\|_{C^2[0,1/2]}+ \sup_{0<a<\frac{1}{2}} \frac{|h(a)|}{a^2}.$$
Our claim can be precisely formulated: \textit{there exists $\varepsilon>0$ small enough such that for any $0\leq d_0\leq \varepsilon$ the equation
\eqref{eq:self-similar eq} has a unique solution in $X_{d_0}$.}

We now proceed to the proof of the claim.
As a first step we show that $T$ maps $X_{d_0}$ into itself.
First of all, by \eqref{eq:explicit Green function} we have the following expansion for $0<b<a$
$$G(a,b)= \frac{1-a^2}{4a^2}\left( \frac{4}{3} b^3 +O(b^5)\right)- 
\frac{1-b^2}{4} \left(\frac{2}{a} +\frac{1-a^2}{a^2} \log \left( \frac{1-a}{1+a} \right) \right).$$
Hence, if we consider $\frac{G(a,b)}{b^2(1-b^2)}$, we get for $0<b<a$
$$\frac{G(a,b)}{b^2(1-b^2)}= \frac{1-a^2}{4a^2}\left( \frac{4}{3} b +O(b^3)\right)- 
\frac{1}{4b^2} \left(\frac{2}{a} +\frac{1-a^2}{a^2} \log \left( \frac{1-a}{1+a} \right) \right).$$
While the first term in $\frac{G(a,b)}{b^2(1-b^2)}$ is clearly analytic for $0<b<a$, the second term exhibits a singularity of the type
$-1/b^2$ near $b=0$: nonetheless when dealing with $T$ we can get rid of this singularity 
by integrating it against $[\sin (2f(b)) - 2f(b)]$ for $f \in X_{d_0}$. \\

Let us now consider a function $f\in X_{d_0}$ which we can write as $f(b)= d_0 \varphi_0(b) +h_f(b)$: since $f$ satisfies by definition the initial condition $f(0)=0$
we are allowed to expand $ H(f(b))(b^2(1-b^2))$ in a neighborhood of $b=0(=f(0))$ in the following standard way
\begin{align}\label{eq:expansion of sinus}
  H(f(b))(b^2(1-b^2))&=\notag \\
  \sin (2f(b)) - 2f(b)&= -\frac{4}{3} f(b)^3+\frac{4}{15} f(b)^5 + O(f(b)^7)
\end{align}
Moreover, for $0\leq b\leq 1/2$, any $f\in X_{d_0}$ satisfies $|f(b)| \leq c d_0 b + d_0^2 b^2\leq M d_0 b $ for an absolute constant $M>0$
and if $d_0$ is small enough. This bound on $f(b)$ 
together with \eqref{eq:expansion of sinus} gives the key estimate for killing the singularity in the second term of $\frac{G(a,b)}{b^2(1-b^2)}$.
Indeed, thanks to all previous considerations, if $f\in X_{d_0}$ we can bound
$$h(a):= \int_0^a G(a,b) H(f(b)) db= \int_0^a \frac{G(a,b)}{b^2(1-b^2)} [\sin (2f(b)) - 2f(b)] db$$
for  $0\leq a\leq 1/2$ as follows
\begin{align} \label{eq:estimates for h}
 &|h(a)|\leq M^3 d_0^3a^2\ll d_0^2 a^2, \notag \\
 &|h'(a)|\leq M^3 d_0^3 a\ll d_0^2 a, \notag \\
 &|h''(a)| \leq M^3 d_0^3 \ll d_0^2,
\end{align}
provided $d_0$ is small enough, i.e. provided $\varepsilon>0$ is chosen sufficiently small.
This proves that $T:X_{d_0} \rightarrow X_{d_0}$.\\
To prove the claim it remains to show that $T$ is indeed a contraction on $X_{d_0}$ with respect to its norm.
We consider $f,g\in X_{d_0}$ and using similar arguments as above we can estimate
\begin{equation*}
 {\|Tf- Tg\|}_{X_{d_0}}\leq C {d_0}^2 {\|f- g\|}_{X_{d_0}}
\end{equation*}
which implies that $T$ is a contraction for $d_0$ sufficiently small, whence the claim.\\
In order to conclude for the regularity of the solution we can argue similarly as Krieger and Schlag did in \cite[Lemma 2.1]{krsc}.
\end{proof}

The next step consists in solving equation \eqref{eq:self-similar eq} backwards starting from $a=1$

\begin{lem}\label{l:lemma2}
There exists $\varepsilon>0$ small enough such that, for any $d_1, d_2 \in (-\varepsilon, \varepsilon)$, the equation \eqref{eq:self-similar eq}
admits a unique solution on $[1/2,1)$ of the form
\begin{equation}\label{eq:solution backwards from 1}
 Q(a)= d_1 \varphi_1(a) + d_2 \varphi_2(a) + (d_3- d_2) \frac{2}{a}+ Q_1(a)
\end{equation}
where $d_3$ is given by $\sin (4d_3)= 4d_2$ and with
\begin{align}\label{eq:asymptotics}
 \varphi_1(a)&=\frac{1-a^2}{a^2}= O(1-a), \notag \\
 \varphi_2(a)&= \frac{2}{a}+ \frac{1- a^2}{a^2}\log \left(\frac{1-a}{1+a}\right)=2\Big(1+O(1-a)\Big)+2 O\Big((1-a) \log (1-a)\Big),\notag \\
 Q_1(a)&=\Big(|d_1|^3+|d_2|^3\Big) O\Big((1-a)^2 \log^2 (1-a) \Big)
\end{align}
where the expansions hold for $a\in [1/2,1]$.
Moreover for $a=1/2$ we have
\begin{align}\label{eq:solution2 in 1/2}
Q(1/2)&= d_1 \varphi_1(1/2) + d_2 \varphi_2(1/2) +4(d_3- d_2) +O \Big(|d_1|^3+|d_2|^3\Big)\notag \\
Q'(1/2)&=d_1 \varphi_1'(1/2) + d_2 \varphi_2'(1/2) -8(d_3- d_2) +O \Big(|d_1|^3+|d_2|^3\Big). 
\end{align}
\end{lem}

\begin{proof}
 Similarly as we have done in the proof of Lemma \ref{l:lemma1}, we argue here by contraction arguments.
 The ansatz for $Q$ solution of \eqref{eq:self-similar eq} is given in integral form as
\begin{equation}\label{eq:ansatz2}
 Q(a)= Q_0(a)+ (d_3- d_2) \frac{2}{a}- \int_a^1 G(a,b) \frac{\sin(2Q(b)) - 2Q_0(b)- 2Q_1(b)}{b^2(1-b^2)} db
\end{equation}
with $Q_0$ defined as
$$Q_0(a):=d_1 \varphi_1(a) + d_2 \varphi_2(a)$$
and where $G$ is the Green function defined in \eqref{eq:explicit Green function}.
Assuming $Q$ has the form \eqref{eq:solution backwards from 1}, then the integral equation \eqref{eq:ansatz2} can be viewed as an equation for $Q_1$ that can be solved by contraction.
The asymptotics \eqref{eq:asymptotics} are a consequence of the behavior of $\frac{G(a,b)}{b^2(1-b^2)}$ for $1/2\leq a<1$ and $a< b<1$.
Indeed we can write $\frac{G(a,b)}{b^2(1-b^2)}$ as the sum of the following two terms
\begin{align} \label{eq:G1 and G2}
 G_1(a,b):&= \frac{1-a^2}{4a^2}\left( \frac{2}{b(1-b^2)} +\frac{1}{b^2} \log \left( \frac{1-b}{1+b}\right)\right)\notag\\
 G_2(a,b):&=- \frac{1}{4b^2} \left(\frac{2}{a} +\frac{1-a^2}{a^2} \log \left( \frac{1-a}{1+a} \right) \right)
\end{align}
so we clearly see that for $1/2\leq a<1$ and $a< b<1$
\begin{align}
 G_1(a,b)\approx (1-a)\left( \frac{1}{(1-b)} + \log \left( 1-b \right)\right)\notag\\
 G_2(a,b)\approx - \left(2 + 2(1-a) \log \left( 1-a \right) \right).
\end{align}
Now, integrating these two terms against $\Big[\sin(2Q(b)) - 2Q_0(b)- 2Q_1(b)\Big]$ which is
of order $O\Big((1-b)\log(1-b)\Big )$ and using the fact that $\sin (4d_3)= 4d_2$ allows to conclude that the integral in the ansatz \eqref{eq:ansatz2} decays like
$\Big((1-a)^2 \log^2 (1-a)\Big)$.
\end{proof}

We remark that on the interval $(1/2,1)$ we obtained a $2$-parameter family ($d_3$ indeed is determined by $d_2$) of solutions: this allows us to 
solve the non--linear connection problem at $a=1/2$, i.e. to ``glue'' smoothly at $a=1/2$ the solutions of Lemma \ref{l:lemma1} and
those of Lemma \ref{l:lemma2}.

\begin{cor}\label{c:connection at 1/2}
Given any $d_0$ small enough, the ordinary differential equation \eqref{eq:self-similar eq} admits a unique $C^2$ self--similar solution $Q(a)$
on the interval $[0,1)$ with initial conditions \eqref{eq:initial conditions}.
In a left neighborhood of $a=1$ this solution $Q$ has the form \eqref{eq:solution backwards from 1}, i.e. it behaves as follows
\begin{equation} \label{eq:left solution in 1}
 Q(a)= d_1  O(1-a) + 2 d_2  O\Big((1-a) \log (1-a)\Big) + d_3+ Q_1(a)
\end{equation}
with $Q_1$ as in \eqref{eq:asymptotics} ($Q_1(1)=0$).
\end{cor}
\begin{proof}
 In order to prove this corollary of Lemmas \ref{l:lemma1} and \ref{l:lemma2} we simply apply the inverse function theorem.
 More precisely, given any $d_0$ small enough, by Lemma \ref{l:lemma2} we can associate to it the unique solution $Q$.
 Now, we aim at finding $d_1$ and $d_2$ small such that \eqref{eq:solution2 in 1/2} match
 the values of $Q$ and $Q'$ at $a=1/2$, i.e. \eqref{eq:solution in 1/2}.
 Since the Jacobian determinant of \eqref{eq:solution2 in 1/2} viewed as functions of $d_1, d_2$ in exactly the Wronskian of $\varphi_1$, $\varphi_2$,
 we easily see that its value at $(d_1, d_2)=0$ is $1$. It being non-zero, we can invoke the inverse function theorem and find the desired $d_1, d_2$ small.
 \end{proof}
Let us finally note that the obtained solution just fails logarithmically to be in $\dot{H}^{\frac{3}{2}}$.

The next step consists in solving the self--similar ODE \eqref{eq:self-similar eq} in the exterior light cone, i.e. for
$a=r/t>1$. Similarly as for the interior light cone, we first solve the problem on the two intervals $(1, 2]$ and $[2, \infty)$ and then we will glue
the two solutions at the connection point $a=2$.

\begin{lem}\label{l:lemma3}
There exists $\varepsilon>0$ small enough such that, for any $\wt{d}_1, \wt{d}_2 \in (-\varepsilon, \varepsilon)$, the equation \eqref{eq:self-similar eq}
admits a unique solution on $(1,2]$ of the form
\begin{equation}\label{eq:solution forward from 1}
 Q(a)= \wt{d}_1 \wt{\varphi}_1(a) + \wt{d}_2 \ti{\varphi}_2(a) - (\wt{d}_3- \wt{d}_2) \frac{2}{a}+ \wt{Q}_1(a)
\end{equation}
where $\wt{d}_3$ is given by $\sin (4\wt{d}_3)= 4\wt{d}_2$ and with
\begin{align}\label{eq:asymptotics 3}
 \wt{\varphi}_1(a)&=\frac{a^2-1}{a^2}= O(a-1), \notag \\
 \wt{\varphi}_2(a)&= -\frac{2}{a}+ \frac{a^2-1}{a^2}\log \left(\frac{a-1}{a+1}\right)=2\Big(-1+O(a-1)\Big)+2 O\Big((a-1) \log (a-1)\Big),\notag \\
 \wt{Q}_1(a)&=\Big(|\wt{d}_1|^3+|\wt{d}_2|^3\Big) O\Big((a-1)^2 \log^2 (a-1) \Big)
\end{align}
where the expansions hold for $a\in (1,2]$.
Moreover for $a=2$ we have
\begin{align}\label{eq:solution3 in 2}
Q(2)&= \wt{d}_1 \wt{\varphi}_1(2) + \wt{d}_2 \wt{\varphi}_2(2) -4(\wt{d}_3- \wt{d}_2) +O \Big(|\wt{d}_1|^3+|\wt{d}_2|^3\Big)\notag \\
Q'(2)&=\wt{d}_1 \wt{\varphi}_1\,'(2) + \wt{d}_2 \wt{\varphi}_2\,'(2) +8(\wt{d}_3- \wt{d}_2) +O \Big(|\wt{d}_1|^3+|\wt{d}_2|^3\Big). 
\end{align}
\end{lem}

We omit the proof of Lemma \ref{l:lemma3} since it can be carried out exactly along the line of the proof of Lemma 2 \ref{l:lemma2}

\begin{lem}\label{l:lemma4}
There exists $\varepsilon>0$ small enough such that, for any $q_1, q_2 \in (-\varepsilon, \varepsilon)$, the equation \eqref{eq:self-similar eq}
admits a unique solution on $[2,\infty)$ which , for $a\rightarrow \infty$, has the form
\begin{equation}\label{eq:solution at infinity}
 Q(a)= q_1 \wt{\varphi}_1(a) + q_2 \ti{\varphi}_2(a)+ O\left(\frac{1}{a^2}\right)
\end{equation}
with
\begin{align} \label{eq:expansion at infinity}
 \wt{\varphi}_1(a)&=\frac{a^2-1}{a^2}, \notag \\
 \wt{\varphi}_2(a)&= -\frac{2}{a}+ \frac{a^2-1}{a^2}\log \left(\frac{a-1}{a+1}\right)= -\frac{4}{a}+O\left(\frac{1}{a^2}\right), \notag \\
\end{align}
where the last expansion holds for $a\rightarrow \infty$.
Moreover for $a=2$ we have
\begin{align}\label{eq:solution4 in 2}
Q(2)&= q_1 \wt{\varphi}_1(2) + q_2 \wt{\varphi}_2(2) +O \Big(|q_1|^3+|q_2|^3\Big)\notag \\
Q'(2)&=q_1 \wt{\varphi}_1\,'(2) + q_2 \wt{\varphi}_2\,'(2) +O \Big(|q_1|^3+|q_2|^3\Big). 
\end{align}
\end{lem}
\begin{proof}
As in the proof of Lemma \ref{l:lemma1} we make use of the Green function, now defined in terms of $\wt{\varphi}_1$ and $ \wt{\varphi}_2$, i.e.
\begin{equation}\label{eq:Green function 2}
 \wt{G}(a,b):= \frac{\wt{\varphi}_1(a)\wt{\varphi}_2(b)- \wt{\varphi}_1(b)\wt{\varphi}_2(a)}{\wt{W}(b)},
\end{equation}
where the Wronskian $\wt{W}$ is now defined as
\begin{equation}\label{eq:Wronskian 2}
 \wt{W}(b):=\wt{\varphi}_1(b) \wt{\varphi}_2'(b)- \wt{\varphi}_1'(b) \wt{\varphi}_2(b).
\end{equation}
Similarly to the computation in the proof of Lemma \ref{l:lemma1} we obtain
\begin{align}\label{eq:explicit Green function 2}
 \wt{W}(b)&= \frac{4}{b^2}, \notag \\
 \wt{G}(a,b)&=\frac{a^2-1}{4a^2}\left( -2b +(b^2-1) \log \left( \frac{b-1}{b+1}\right)\right)- \frac{b^2-1}{4} \left(-\frac{2}{a} +\frac{a^2-1}{a^2} \log \left( \frac{a-1}{a+1} \right) \right). 
\end{align}
In this case, the natural ansatz for $Q$ is as follows
\begin{equation}\label{eq:ansatz3}
 Q(a)= q_1 \wt{\varphi}_1(a) + q_2 \ti{\varphi}_2(a) +\int_a^\infty \wt{G}(a,b) \frac{\sin(2Q(b))- 2Q(b)}{b^2(1-b^2)} db.
\end{equation}
In order to solve \eqref{eq:self-similar eq} on $[2,\infty)$ we can set up a contraction argument for the ansatz \eqref{eq:ansatz3}
in analogy to the proof of Lemma \ref{l:lemma1}.
Let us note that we have the following
\begin{align*}
 \frac{\wt{G}(a,b)}{b^2(1-b^2)}&= \frac{a^2-1}{4a^2}\left(\frac{2}{b(b^2-1)}-\frac{1}{b^2}\log \left(\frac{b-1}{b+1}\right)\right)\\&+
\frac{1}{4b^2}\left(-\frac{2}{a}+\frac{a^2-1}{a^2}\log\left(\frac{a-1}{a+1}\right)\right)
\end{align*}
which to leading order (and up to constants) for $a, b\rightarrow \infty$ exhibits the following decay
\begin{equation*}
 \frac{\wt{G}(a,b)}{b^2(1-b^2)}\approx \frac{1}{b^3}-\frac{1}{a b^2}.
\end{equation*}
Integrating this against $\Big[\sin(2Q(b))- 2Q(b)\Big]$ for $a<b<\infty$ gives the asymptotics \eqref{eq:solution at infinity}. The values \eqref{eq:solution4 in 2}
are simply obtained by substituting $a=2$.
\end{proof}

In view of Lemmas \ref{l:lemma3} and \ref{l:lemma4} which both generate a two-parameter family of solutions we can glue
those at $a=2$ so to produce a one smooth solution for all $a>1$. This is done in the following corollary.

\begin{cor}\label{c:connection at 2}
Given any $q_1, q_2$ small enough, the ordinary differential equation \eqref{eq:self-similar eq} admits a unique $C^2$ self--similar solution $Q(a)$
on the interval $(1,\infty)$ with the asymptotics \eqref{eq:solution at infinity}-\eqref{eq:expansion at infinity}.
In a right neighborhood of $a=1$ the solution $Q$ has the form \eqref{eq:solution forward from 1}.
\end{cor}
\begin{proof}
 The corollary can be proven once more thanks to the inverse function theorem.
 Given $q_1, q_2$ small enough, Lemma \ref{l:lemma4} provides a unique solution $Q$ for $a\in [2, \infty)$ 
 with asymptotics \eqref{eq:solution at infinity}-\eqref{eq:expansion at infinity}.
 Now, we aim at finding $\wt{d}_1$ and $\wt{d}_2$ small ($\wt{d}_3$ is given as function of $\wt{d}_2$) such that \eqref{eq:solution3 in 2} given by Lemma \ref{l:lemma3} match
 the values of $Q$ and $Q'$ at $a=2$, i.e. \eqref{eq:solution4 in 2}.
 Since the Jacobian determinant of \eqref{eq:solution3 in 2}, i.e. the Wronskian, does not vanish
 and by smallness of $q_1$ and $q_2$,
 we can apply the inverse function theorem in a neighborhood of $(\wt{d}_1, \wt{d}_2)=0$ and 
 solve the connection problem at $a=2$.
 \end{proof}
 
 Finally, to complete the construction of the self--similar solution we will connect just continuously the solution on $[0,1)$ provided
 by Corollary \ref{c:connection at 1/2} with the solution on $(1, \infty)$ provided by Corollary \ref{c:connection at 2}.
 
 \begin{cor}\label{c:matching at light cone}
 Given any small $d_0$, there exists a unique $C^2$ solution $Q(a)$ of the ordinary differential equation \eqref{eq:self-similar eq}
 on $[0,1)$ with initial conditions \eqref{eq:initial conditions}. By Corollary \ref{c:connection at 2}, this solution can be extended 
 non-uniquely to a continuous function on $a\geq 1$ which solves
 \eqref{eq:self-similar eq} on $a>1$ and behaves as
 \begin{equation} \label{eq:first behavior at infty}
 Q(a) = q_1 - 4 q_2 a^{-1}+ O(a^{-2}) 
 \end{equation}
for $a\rightarrow \infty$. The global continuous solutions on $a\geq0$ must satisfy the condition $2d_3= -2 \wt{d}_3$ (see Lemmas \ref{l:lemma2} and \ref{l:lemma3}).
Let us denote these solutions on $a>0$ by $Q_0(a)$, then we have the following global representation
\begin{equation}\label{eq:global representation}
 Q_0(a)= C_1 \frac{|a^2-1|}{a^2} +C_2 \frac{|a^2-1|}{a^2} \log\left(\frac{|a-1|}{a+1}\right) +C_3 Q_3(a)+Q_4(a)
\end{equation}
for $a\geq 1/2$, where $Q_3= \frac{1}{a}$ is smooth at $a=1$ and $Q_4(a)=O(|a-1|^2\log^2(|a-1|))$ consists of higher order terms.

\end{cor}
\begin{proof}
  Given any small $d_0$, Corollary \ref{c:connection at 1/2} provides us with $d_1, d_2, d_3$ and with a $C^2$ solution $Q$
  on $[0,1)$. On the other hand, thanks to Corollary \ref{c:connection at 2} any $q_1, q_2$ small enough give
  $\wt{d}_1$, $\wt{d}_2$ and $\wt{d}_3$ corresponding to a $C^2$ solution on $(1, \infty)$. The freedom of selecting $q_1$ and $q_2$ small allows us
  to make the choice in such a way that $\wt{d}_3$ satisfies $-2 \wt{d}_3=2d_3$, i.e. such that we can extend continuously $Q$ to $a>1$.
\end{proof}

\section{Large self--similar solutions} \label{sec:large}
In this section we show that it is possible to choose a big constant $\wt{d}_1$ in Lemma \ref{l:lemma3} thus allowing to construct
solutions which are arbitrarily large in the exterior light-cone. Indeed, by Corollary \ref{c:matching at light cone}, the continuity condition at
the light cone involves only $\wt{d_3}$ (and hence $\wt{d}_2$) leaving freedom of choice for $\wt{d}_1$. Hence, the goal here is to prove the existence
of a solution to the ODE \eqref{eq:self-similar eq} on $a>1$ for $\wt{d}_1$ large.
The first step consists in proving the analogue of Lemma \ref{l:lemma3} on a right neighborhood of $a=1$ for arbitrary $\wt{d_1}$.

\begin{lem}\label{l:lemma5}
There exists $\varepsilon>0$ small enough such that, for any $\wt{d}_2 \in (-\varepsilon, \varepsilon)$ and for any $\wt{d}_1\geq 1$ arbitrary,
the equation \eqref{eq:self-similar eq}
admits a unique solution $Q$ on $(1,1+\ell]$ with $\ell=c \wt{d}_1 \, ^{-\frac{1}{2}}$ for some absolute constant $c>0$ small enough. This solution $Q$ has the form
\begin{equation}\label{eq:large solution forward from 1}
 Q(a)= \wt{d}_1 \wt{\varphi}_1(a) + \wt{d}_2 \ti{\varphi}_2(a) - (\wt{d}_3- \wt{d}_2) \frac{2}{a}+ \wt{Q}_1(a)
\end{equation}
where $\wt{d}_3$ is given by $\sin (4\wt{d}_3)= 4\wt{d}_2$ (and hence $\wt{d}_3$ is also small) and with
\begin{align}\label{eq:asymptotics 5}
 \wt{\varphi}_1(a)&=\frac{a^2-1}{a^2}= O(a-1), \notag \\
 \wt{\varphi}_2(a)&= -\frac{2}{a}+ \frac{a^2-1}{a^2}\log \left(\frac{a-1}{a+1}\right)=2\Big(-1+O(a-1)\Big)+2 O\Big((a-1) \log (a-1)\Big),\notag \\
 \wt{Q}_1(a)&=\wt{d}_1 O\Big((a-1)^{2} \log^2 (a-1) \Big)
\end{align}
where the expansions hold for $a\in (1,1+\ell]$.
Moreover, there exists $a_* \in (1, 1+\ell]$ such that
\begin{equation}\label{eq:large value}
 |Q(a_*)|\approx \wt{d}_1 \;^{\frac{1}{2}}.
\end{equation}
As a consequence $Q(a_*)$ can be made arbitrarily large by choosing $\wt{d}_1$ sufficiently large.
\end{lem}
\begin{proof}
 The key ideas are inherited from the proof of Lemma \ref{l:lemma2} except the fact that we have to deal here with lack of smallness for $\wt{q}_1$:
 this is overcome by a bootstrap argument.
 
 The ansatz for $Q$  is given in integral form as
\begin{equation}\label{eq:ansatz5}
 Q(a)= \wt{Q}_0(a)- (\wt{d}_3- \wt{d}_2) \frac{2}{a}+ \int_1^a \wt{G}(a,b) \frac{\sin(2Q(b)) - 2\wt{Q}_0(b)- 2\wt{Q}_1(b)}{b^2(1-b^2)} db
\end{equation}
with $\wt{Q}_0$ defined as
$$Q_0(a):=\wt{d}_1 \wt{\varphi}_1(a) + \wt{d}_2 \wt{\varphi}_2(a)$$
and where $\wt{G}$ is the Green function defined in \eqref{eq:Green function 2}.
Assuming $Q$ has the form \eqref{eq:large solution forward from 1}, the integral equation \eqref{eq:ansatz5} can be solved by contraction for $\wt{Q}_1$.
More precisely, for $Q$ as in \eqref{eq:large solution forward from 1}, we would like to obtain $\wt{Q}_1$ as fixed point of the following equation
\begin{equation} \label{eq:fixed point problem}
 \wt{Q}_1(a)= \int_1^a \wt{G}(a,b) \frac{\sin(2Q(b)) - 2\wt{Q}_0(b)- 2\wt{Q}_1(b)}{b^2(1-b^2)} db.
\end{equation}

For the sake of clarity let us define
$$\bar{Q}_1(a):= \frac{\wt{Q}_1(a)}{(a-1)^{2} \log^2 (a-1)}$$
In order to run the fixed point argument for equation \eqref{eq:fixed point problem} we need to show that the bound 
\begin{equation} \label{eq:first bound}
 |\bar{Q}_1(a)|\leq  C |\wt{d}_1|
\end{equation}
improves upon itself on the interval $a \in [1, 1 + \ell]$, with $\ell=c \wt{d}_1 \, ^{-\frac{1}{2}}$ as given by the statement of the lemma,
if inserted in the equation \eqref{eq:fixed point problem}. More precisely, our goal is to prove that the bound \eqref{eq:first bound}
entails the following better one
$$|\bar{Q}_1(a)|\leq  \frac{C}{2} |\wt{d}_1|.$$
As a first step, we compute $\frac{\wt{G}(a,b)}{b^2(1-b^2)}$ and we easily obtain 
\begin{align*}
 \frac{\wt{G}(a,b)}{b^2(1-b^2)}&= \frac{a^2-1}{4a^2}\left(\frac{2}{b(b^2-1)}-\frac{1}{b^2}\log \left(\frac{b-1}{b+1}\right)\right)\\&+
\frac{1}{4b^2}\left(-\frac{2}{a}+\frac{a^2-1}{a^2}\log\left(\frac{a-1}{a+1}\right)\right).
\end{align*}
Similarly to \eqref{eq:G1 and G2}, we can write $\frac{\wt{G}(a,b)}{b^2(1-b^2)}$ as the sum of two terms $\wt{G}_1$ and $\wt{G}_2$ which behave
as follows for $1<a\leq 1+\ell$ and $1< b<a$
\begin{align} \label{eq:tilde G1 and G2}
 \wt{G}_1(a,b)\approx (a-1)\left( \frac{1}{(b-1)} + \log \left( b-1 \right)\right)\notag\\
 \wt{G}_2(a,b)\approx - \left(-2 + 2(a-1) \log \left( a-1 \right) \right).
\end{align}
The choice of $l$ is such that on $[1, 1 + \ell]$ the non-linearities involving $\wt{Q}_1$ in $\Big[\sin(2Q(b)) - 2\wt{Q}_0(b)- 2\wt{Q}_1(b)\Big]$
are dominated by their first order (linear) approximation.
We also introduce $\bar{Q}_0(a):=\wt{Q}_0(a)- (\wt{d}_3- \wt{d}_2) \frac{2}{a}$ so that $Q(a)= \bar{Q}_0(a)+\wt{Q}_1(a)$.
Now, we can write $\sin(2Q(b)) - 2\wt{Q}_0(b)- 2\wt{Q}_1(b)$ as follows
\begin{align} \label{eq:nonlinearities}
\sin(2Q) - 2\wt{Q}_0- 2\wt{Q}_1 &= \Big( \sin (2\bar{Q}_0) \cos(2\wt{Q}_1)- 2 \wt{Q}_0\Big)\notag \\
&+ \Big( \cos (2\bar{Q}_0) \sin(2\wt{Q}_1)- 2 \wt{Q}_1\Big)\notag \\
&=: A+B.
\end{align}
The condition $\sin (4\wt{d}_3)= 4\wt{d}_2$ already ensures that $\wt{Q}_1(1)=0$.
Now, we can expand $\sin$ and $\cos$ in \eqref{eq:nonlinearities} and integrate against $\wt{G}_1$ and $\wt{G}_2$,

At first we consider $A$. With the choice of $\ell=c \wt{d}_1 \, ^{-\frac{1}{2}}$ for some small constant $c>0$
and under the assumption \eqref{eq:first bound} on $\wt{Q}_1$, $\cos(2\wt{Q}_1)$ can be approximated by $1$, being the higher order terms dominated by the first order one; as a consequence
$A$ can be further expanded as follows (remember that $\sin (4\wt{d}_3)= 4\wt{d}_2$)
\begin{align} \label{eq:A}
A&\approx \Big[\sin(-4\wt{d}_3) \cos(\beta)+ \cos(-4\wt{d}_3) \sin (\beta) \Big] \cos(2\wt{Q}_1)- \beta +4\wt{d}_2 \notag  \\
&\approx \left[-4\wt{d}_2 \left(1 + (\cos(\beta)-1) \right)+ \left(\beta+ (\sin(\beta)- \beta)\right)\right] - \beta +4\wt{d}_2 \notag   \\
&\approx \left[-4\wt{d}_2 \left(\cos(\beta)-1\right)+  \left( \sin(\beta)- \beta\right) \right]  
\end{align}
where $\beta\approx 4 \wt{d}_1 (b-1) +4 \wt{d}_2 (b-1) \log (b-1)$ and where we also approximated $\cos(-4\wt{d}_3)$ by $1$ being $\wt{d}_3$ small.
If we integrate $A$ against $\wt{G}_1$, in particular against the first and leading term of $\wt{G}_1$ i.e. $(a-1)/(b-1)$ we obtain
\begin{align} \label{eq:integrate A and G1}
\wt{d}^{-1} &(a-1)^{-2} \left| \int_1^a \wt{G}_1(b)\, A \, db \right| \notag \\
&\lesssim \wt{d}^{-1} \ell^{-2} (a-1) \int_1^a \frac{1}{b-1}\,  \left|-4\wt{d}_2 \left(\cos(\beta)-1\right)+  \left( \sin(\beta)- \beta\right)  \right| \, db  \notag \\
&\lesssim c^ {-2} (a-1) \int_1^a \,  \left(4|\wt{d}_2|+ 2 \right)\left|\frac{\beta}{b-1} \right|  \, db  \notag \\ 
&\lesssim     \left(4|\wt{d}_2|+ 2 \right) \left((\wt{d}_1+|\wt{d}_2|)\ell^2+ |\wt{d}_2| \ell^2 |\log \ell| \right) \ll 1
\end{align}
by choice of $\ell=c \wt{d}_1 \, ^{-\frac{1}{2}}$.
Similarly we can also estimate the term corresponding to the integral of $A$ against $\wt{G}_2$.
As far as the integral involving $B$ is concerned we argue as follows.
Using the fact that $\cos(2\bar{Q}_0)$ is bounded and that $\sin(2\wt{Q}_1)$ behaves to leading order as $(2\wt{Q}_1)$ 
due to the choice of $\ell=c \wt{d}_1 \, ^{-\frac{1}{2}}$ and to the assumption \eqref{eq:first bound}, we can estimate
\begin{align} \label{eq:integrate B and G1}
\wt{d}^{-1} &(a-1)^{-2} \left| \int_1^a \wt{G}_1(b)\, B \, db \right| \notag \\
&\leq  \wt{d}^{-1} \ell^{-2} \left| \int_1^a \wt{G}_1(b)\, \Big( \cos (2\bar{Q}_0(b)) \sin(2\wt{Q}_1(b))- 2 \wt{Q}_1(b)\Big) \, db \right|\notag \\
&\lesssim c^ {-2}\left| \int_1^a \wt{G}_1(b)\, \Big[ \left(\cos (2\bar{Q}_0(b))-1\right) \sin(2\wt{Q}_1(b))+ \left(\sin(2\wt{Q}_1(b)) - 2 \wt{Q}_1(b)\right)\Big] \, db \right| \notag \\ 
&\lesssim  c \wt{d}^{-\frac{1}{2}} \ll 1
\end{align}
and an analogous estimate can be obtained when integrating $B$ against $\wt{G}_2$.
The large value exhibited in \eqref{eq:large value} is achieved by choosing $a_*=1+ \ell/2$.
\end{proof}

\begin{lem}
 The solutions to \eqref{eq:self-similar eq} provided by Lemma \ref{l:lemma5} on $(1, 1+l]$ can be extended to $(1,\infty)$ as a smooth globally bounded solution $Q(a)$, which behaves as 
\begin{equation}\label{eq:behavior at infty}
 Q(a) = c_1 + c_2 a^{-1}+ O(a^{-2})
\end{equation}
for $a\rightarrow \infty$ and for non--vanishing constants $c_1$ and $c_2$.
\end{lem}
\begin{proof}
Let us recall that we are dealing with the following ODE (\eqref{eq:self-similar eq})
\begin{equation}
 Q''(a) + \frac{2}{a} Q'(a)-\frac{\left( 2 \sin Q(a)  \cos Q (a) \right)}{a^2(1-a^2)} = 0,
\end{equation}
where the non--linearity $\left(2 \sin Q(a)  \cos Q(a)\right)$ is clearly bounded. Hence, away from the singularities $a=0$ and $a=1$, standard elliptic estimates allow to prove 
$L^\infty$ bounds independent of the existence time so that solutions can be extended for all time.

\end{proof}

\section{Excision of singularity near light cone $a = 1$: approximate solutions} \label{sec:regularization}
In the previous section we have shown the existence of self--similar solutions $Q_0$ to \eqref{eq:self-similar eq}
which are smooth away from the light cone but are only continuous at $a=1$ i.e. across the light cone.
Our goal is to construct global smooth solutions to \eqref{eq:scalar wave equation} which have infinite critical norm $\dot{H}^{\frac{3}{2}}$
departing from these self--similar solutions by excision of the singularity near $a=1$. In order to achieve this, we introduce a smooth cut--off function
$\chi(t-r)$ whose support lies at a distance $C$ from the light cone: more precisely $\chi(x)=1$ for $|x|\geq 2C$ and $\chi(x)=0$ for $|x|\leq C$.
In view of Corollary \ref{c:matching at light cone}, we know that near $a=1$, the function $Q_0$
is of the form
\begin{equation}\label{eq:Q_0 near light cone}
 Q_0(a)= C_1 \frac{|a^2-1|}{a^2} +C_2 \frac{|a^2-1|}{a^2} \log\left(\frac{|a-1|}{a+1}\right) +C_3 Q_3(a)+Q_4(a)
\end{equation}
where $Q_3=\frac{2}{a}$ is smooth at $a=1$ and $Q_4(a)=O(|a-1|^2\log^2(|a-1|))$ consists of higher order terms.
For the sake of simplicity we introduce
\[
 R(t,r)=R(a)= C_1 \frac{|a^2-1|}{a^2} +C_2 \frac{|a^2-1|}{a^2} \log\left(\frac{|a-1|}{a+1}\right).
\]
Thus, we modify the exact singular solution near the light cone by introducing the following approximate solution
\begin{equation}\label{eq:approximate solution}
 u_{approx}(t, r): = \chi(t-r)\left[R(a) +Q_4(a)\right] + C_3 Q_3(a)
\end{equation}
Now, we will estimate the errors which arise when computing the expression 
\[
\partial_{tt} u_{approx}- \partial_{rr} u_{approx}- \frac{2}{r} \partial_r  u_{approx} + \frac{f(u_{approx})}{r^2}.
\]
Since for our estimates, the term $Q_4$ is of higher order we will neglect it in the following computations.
First of all, we remark when applying derivatives only to $\chi$ we get $R( \partial_{tt} \chi-\partial_{tt} \chi  )$ which 
vanishes since $\chi(t-r)$ solves the $1$--dimensional wave equation. 
When only one time derivative falls on $\chi(t-r)$ we get 
\begin{equation*}
 2 \chi'(t-r) \partial_t R(t,r)= 2 \chi'(t-r) \sign(t-r)\left(2C_1 \frac{t}{r^2}-2C_2 \frac{t}{r^2} \log\left(\frac{|r-t|}{r+t}\right) -2C_2 \frac{1}{r}\right).
\end{equation*}
Similarly, when one derivative $\partial_r$ falls on $\chi(t-r)$ we obtain
\[
-2\chi'(t-r)\partial_r R(t,r)= 2 \chi'(t-r)\sign(t-r) \left(2C_1 \frac{t^2}{r^3}+2C_2 \frac{t^2}{r^3} \log\left(\frac{|r-t|}{r+t}\right)+ 2C_2 \frac{t}{r^2}\right).
\]
Finally, we compute the contribution from $-\frac{2}{r}R\partial_r \chi$ which gives
\[
-2\chi'(t-r)\left(C_1 \frac{|r^2-t^2|}{r^3}+C_2 \frac{|r^2-t^2|}{r^3} \log\left(\frac{|r-t|}{r+t}\right)\right).
\]
Summing up all the terms involving $\chi'$ we have the following expression 
\begin{align*}
 2 C_2 \,\chi'(t-r)\log\left(\frac{|r-t|}{r+t}\right) &\left(\frac{t}{r^2}\left|\frac{r}{t}-1\right| - \frac{t^2}{r^3}\left|\frac{r}{t}-1\right|\right)- 2C_2\chi'(t-r)\left(\frac{t}{r^2}\left|\frac{r}{t}-1\right|\right)\\
&+2C_1 \chi'(t-r)\left(\frac{t^2}{r^3}\left|\frac{r^2}{t^2}-1\right|-2\frac{t^2}{r^3}\left|\frac{r}{t}-1\right|\right)
\end{align*}
which is of size $t^{-3}$, since $\chi'$ has support in the strip $C\leq|t-r|\leq 2C$ where $|r/t-1|$ behaves as $1/t$.

Moreover, the error from the nonlinear term is of the following form 
\begin{align*}
 \frac{1}{r^2} f&\left(\chi(t-r)\left(C_1 \frac{|r^2- t^2|}{r^2} +C_2 \frac{|r^2- t^2|}{r^2} \log\left(\frac{|r-t|}{r+t}\right)\right)+ 2C_3\sign(t-r)\frac{t}{r} \right)\\
& - \frac{\chi(t-r)}{r^2}f\left(C_1 \frac{|r^2- t^2|}{r^2} +C_2 \frac{|r^2- t^2|}{r^2} \log\left(\frac{|r-t|}{r+t}\right)+ 2C_3 \sign(t-r) \frac{t}{r} \right)\\
& = O(t^{-3})
\end{align*}
This implies that 
\[
\Box u_{approx} + \frac{f(u_{approx})}{r^2}\in L^2(\R^3),
\]
in light of the support properties of this expression, and is of order $t^{-2}$ at fixed time $t$ in this norm. Thus all the errors beat the scaling.

\section{From an approximate solution to an exact solution}\label{sec:completion}

Here we construct exact solutions via the ansatz 
\[
u(t, r) = u_{approx}(t, r) + \epsilon(t, r).
\]
We recall that we are considering the case of target $\Sf^3$, in which case we have $g(u) = \sin u$. Then we obtain the following wave equation which is in fact on $\R^{5+1}$:
\begin{equation}\label{eq:veqn}\begin{split}
\left(\frac{\epsilon}{r}\right)_{tt} - \left(\frac{\epsilon}{r}\right)_{rr} - \frac{4}{r}\left(\frac{\epsilon}{r}\right)_{r}& = -\frac{1}{r}\frac{\sin(2\epsilon) - 2\epsilon}{r^2}\cos(2u_{approx})\\
&-\frac{\sin(2u_{approx})}{r^3}\left(\cos(2\epsilon) - 1\right)\\
&-\frac{2\epsilon}{r^3}\left[\cos(2u_{approx}) - 1\right]\\
&+\frac{e_0}{r}. 
\end{split}\end{equation}
 Note that by introducing the new variable $v = \frac{\epsilon}{r}$, we get ``essentially'' the new wave equation 
 \[
 v_{tt} - \triangle_{\R^5}v = v^3 + \frac{v u_{approx}^2}{r^2} + \frac{v^2 u_{approx}}{r} + \frac{1}{r}e_0.
 \]
 For the purely cubic term, we get the scaling $v(t, r)\rightarrow \lambda v(\lambda t, \lambda r)$, which on $\R^{5+1}$ corresponds to $s_c = \frac{3}{2}$, as expected.  Thus it is natural to try to run an iteration in the space $H_{\R^5}^{\frac{3}{2}}$. 
It is then natural to work with the Strichartz norms  $L_t^\infty L_x^5\cap L_t^2 L_x^{10}\cap L_t^2(\nabla_x^{-\frac{1}{2}}L_x^5)$, with the same scaling as $L_t^\infty\dot{H}^{\frac{3}{2}}_{\R^5}$. 
Then the interaction terms $\frac{v u_{approx}^2}{r^2}, \frac{v^2 u_{approx}}{r}$, appear critical, since the first can be reduced to 
\[
\frac{v u_{approx}}{r t},
\]
which fails logarithmically to belong to $L_t^1 \dot{H}^{\frac{1}{2}}$. Thus, as in the paper \cite{krsc}, we shall also be taking advantage of the Hamiltonian structure to handle this low frequency issue. 
We state 
\begin{prop}\label{prop:vMain}
Let $C\geq 1$ a given constant, $T\geq 1$ sufficiently large, depending on $\tilde{d}_1$ in the approximate solution. 
Assume that the $C^\infty$-smooth data $v[T] = (v(T, \cdot), v_t(T, \cdot))$ are radial and supported in the annulus $r\in [T-C, T+C]$. 
Also, assume that $\tilde{d}_2$ for the approximate solution is sufficiently small (less than an absolute constant), and that for a $\delta_1>0 = \delta_1(C)$ sufficiently small, we have  
\[
\big\|v[T]\big\|_{\dot{H}^{\frac{3}{2}}(\R^5)\cap \dot{H}^1(\R^5)\times \dot{H}^{\frac{1}{2}}(\R^5)\cap L^2(\R^5)}\leq \delta_1. 
\]
Then the problem \eqref{eq:veqn} with initial data $v[T]$ at time $t = T$ admits a global-in-forward time solution $v(t, \cdot)$ of class $C^\infty$. 
\end{prop}

The existence of the solution $v$ will follow from a standard local existence result as well as a more sophisticated bootstrap argument which is at the heart of the matter, as in \cite{krsc}.
The local existence result is as follows: 
\begin{prop}\label{prop:local exist} 
Given data $v[T]$ with the same support properties as above, satisfying 
\[
\big\|v[T]\big\|_{\dot{H}^{\frac{3}{2}}(\R^5)\times \dot{H}^{\frac{1}{2}}(\R^5)}\ll 1,
\]
then there exists a time $T_1>T$ and a solution of \eqref{eq:veqn} of class 
\[
v\in L_t^\infty \dot{H}^{\frac{3}{2}}([T, T_1]\times \R^5),\,v_t\in  L_t^\infty \dot{H}^{\frac{1}{2}}([T, T_1]\times \R^5)
\]
with compact support on every time slice $t\in [T, T_1]$. If $v[T]\in \dot{H}^s\cap \dot{H}^{s-1}$, $s>\frac{3}{2}$, then so is the solution at all times $t\in [T, T_1]$. 
\end{prop}
The proof is standard and we refer to \cite[Section 7]{krsc} for a similar argument. 
\\

Before stating the key bootstrap proposition, we recall the following set of standard Strichartz estimates, in the $\R^{5+1}$-setting: 
\begin{lem} Let $\Box_{5+1} u = 0$. Then for $\frac{1}{p} + \frac{2}{q}\leq 1$,  $p\geq 2$, we have ($u[0] = (u(0, \cdot), u_t(0, \cdot))$)
\[
\big\|(-\triangle)^{\frac{\gamma}{2}}u\big\|_{L_t^p L_x^q}\leq C\big\|u[0]\big\|_{\dot{H}^{\frac{3}{2}}\times \dot{H}^{\frac{1}{2}}},\,\gamma = -1 + \frac{1}{p} + \frac{5}{q}. 
\]

\end{lem}

Next, the bootstrap proposition: 

\begin{prop}\label{prop:bootstrap} Let us assume all hypotheses of Proposition~\ref{prop:vMain}. Then there exists $C_1>1$ with $C_1\delta_1\ll 1$, as well as a constant $\gamma = \gamma(\tilde{d}_{1,2}, \delta_1, T)>0$, such that for any $T_1>T$, the following conclusion holds: 
\[
\big\|v\big\|_{L_t^2(L_x^{10}\cap \nabla_x^{-\frac{1}{2}}L_x^5)([T, T_1]\times \R^5)} + \sup_{t\in [T, T_1]}\big\|v[t, \cdot]\big\|_{\dot{H}^{\frac{3}{2}}\cap \left(\frac{t}{T}\right)^{\gamma}\dot{H}^1(\R^5)\times \dot{H}^{\frac{1}{2}}\cap \left(\frac{t}{T}\right)^{\gamma}L^2(\R^5)}\leq C_1\delta_1
\]
implies 
\[
\big\|v\big\|_{L_t^2(L_x^{10}\cap \nabla_x^{-\frac{1}{2}}L_x^5)([T, T_1]\times \R^5)} + \sup_{t\in [T, T_1]}\big\|v[t, \cdot]\big\|_{\dot{H}^{\frac{3}{2}}\cap \left(\frac{t}{T}\right)^{\gamma}\dot{H}^1(\R^5)\times \dot{H}^{\frac{1}{2}}\cap \left(\frac{t}{T}\right)^{\gamma}L^2(\R^5)}\leq \frac{C_1}{2}\delta_1.
\]

\end{prop}
\begin{proof} We follow closely the procedure in \cite{krsc}. We commence with the energy type norm, i. e. 
\[
\sup_{t\in [T, T_1]}\big\|v[t, \cdot]\big\|_{\left(\frac{t}{T}\right)^{\gamma}\dot{H}^1(\R^5)\times \left(\frac{t}{T}\right)^{\gamma}L^2(\R^5)}.
\]
Multiplying \eqref{eq:veqn} by $v_t$ and integrating in space-time, we get 
\begin{equation}\label{eq:enident}\begin{split}
\int_{\R^5}[v_t^2 + |\nabla_x v|^2]\,dx|_{t=T_1} &= \int_{\R^5}[v_t^2 + |\nabla_x v|^2]\,dx|_{t=T}\\& - \int_T^{T_1}\int_{\R^5}\frac{1}{r}\frac{\sin(2vr) - 2vr}{r^2}\cos(2u_{approx})v_t\,dx dt\\
& - \int_T^{T_1}\int_{\R^5}\frac{\sin(2u_{approx})}{r^3}\big(\cos(2vr) - 1\big)v_t\,dx dt\\
& -  \int_T^{T_1}\int_{\R^5}\frac{2v}{r^2}\big[\cos(2u_{approx}) - 1\big] v_t\,dx dt\\
& +  \int_T^{T_1}\int_{\R^5}\frac{e_0}{r}v_t\,dx dt.\\
\end{split}\end{equation}
We treat the terms on the right via integration by parts. For the second term on the right, introducing 
\[
E(x): = \int_0^x (\sin y - y)\,dy,\,|E(x)|\lesssim |x|^4, 
\]
we get 
\begin{align*}
& \int_T^{T_1}\int_{\R^5}\frac{1}{r}\frac{\sin(2vr) - 2vr}{r^2}\cos(2u_{approx})v_t\,dx dt\\
& =\int_{\R^5} r^{-4}E(vr)\cos(2u_{approx})\,dx|_{T}^{T_1} -  \int_T^{T_1}\int_{\R^5}2r^{-4}E(vr)\sin(2u_{approx})u_{approx, t}\,dx dt.\\
\end{align*}
To bound these terms we use interpolation between $L_x^{\frac{10}{3}}$ and $L_x^5$. With 
\[
\frac{1}{4} = \alpha\cdot\frac{3}{10} + (1-\alpha)\cdot\frac{1}{5} = \frac{1}{5} + \frac{\alpha}{10},
\]
we get $\alpha = \frac{1}{2}$, whence 
\begin{align*}
\left|\int_{\R^5} r^{-4}E(vr)\cos(2u_{approx})\,dx|_{T}^{T_1}\right|\lesssim \int_{\R^5} v^4\,dx &\lesssim \sum_{t = T, T_1}\|\nabla_x v(t,\cdot)\|_{L_x^2}^2\left(\int_{\R^5}|v|(t, \cdot)^5\,dx\right)^{\frac{2}{5}}\\
&\ll \sup_{t\in [T, T_1]}\|\nabla_x v(t,\cdot)\|_{L_x^2}^2,
\end{align*}
where in the last step we have used the bootstrap assumption. This can then be easily absorbed on the left hand side in \eqref{eq:enident}. \\
On the other hand, for the space time integral, we have 
\begin{align*}
&\left|\int_T^{T_1}\int_{\R^5}2r^{-4}E(vr)\sin(2u_{approx})u_{approx, t}\,dx dt\right|\\
&\lesssim \int_T^{T_1}\int_{\R^5}t^{-1} v^4\,dx dt\leq \sup_{t\in [T, T_1]}\left(\int_{\R^5}|v|(t, \cdot)^5\,dx\right)^{\frac{2}{5}}\int_T^{T_1}t^{-1}\|\nabla_x v(t,\cdot)\|_{L_x^2}^2\,dt.
\end{align*}
To bound this last term, using the bootstrap assumption, we have (with an absolute implied constant independent of all other constants)
\begin{align*}
&\sup_{t\in [T, T_1]}\left(\int_{\R^5}|v|(t, \cdot)^5\,dx\right)^{\frac{2}{5}}\int_T^{T_1}t^{-1}\|\nabla_x v(t,\cdot)\|_{L_x^2}^2\,dt\\
&\lesssim (C_1\delta_1)^2 \int_T^{T_1}t^{-1}(C_1\delta_1)^2\left(\frac{t}{T}\right)^{2\gamma}\,dt\leq \frac{(C_1\delta_1)^2}{2\gamma}(C_1\delta_1)^2\left(\frac{T_1}{T}\right)^{2\gamma}, 
\end{align*}
which suffices for the bootstrap, provided $\delta_1^2\ll \gamma$. This deals with the second term on the right hand side of \eqref{eq:enident}. To deal with the third term, write $F(x): = \int_0^x[\cos x - 1]\,dx$, 
whence $|F(x)|\lesssim |x|^3$. Then we obtain 
\begin{align*}
&\int_T^{T_1}\int_{\R^5}\frac{\sin(2u_{approx})}{r^3}\left(\cos(2vr) - 1\right)v_t\,dx dt\\
& = \int_{\R^5}r^{-4}F(vr)\sin(2u_{approx})\,dx|_{T}^{T_1}\\
& - \int_T^{T_1}\int_{\R^5}r^{-4}F(vr)\cos(2u_{approx})\cdot 2u_{approx, t}\,dxdt.
\end{align*}
Here we have 
\begin{align*}
\left| \int_{\R^5}r^{-4}F(vr)\sin(2u_{approx})\,dx|_{T}^{T_1}\right|
&\leq \sum_{t = T, T_1} \int_{\R^5}\frac{|v^3(t, \cdot)|}{r}\,dx\\
&\leq 2\sup_{t\in [T, T_1]}\left\|\frac{v(t, \cdot)}{r}\right\|_{L_x^2}\left\|v^2(t, \cdot)\right\|_{L_x^2}\\
&\lesssim 2\sup_{t\in [T, T_1]}\left\|\nabla_x v(t, \cdot)\right\|_{L_x^2}^2\left\|v(t, \cdot)\right\|_{L_x^5}\\
&\ll \sup_{t\in [T, T_1]}\left\|\nabla_{x} v\right\|_{L_x^2}^2,
\end{align*}
where we have used Hardy's inequality and the bootstrap assumption. This can again be absorbed on the left hand side of \eqref{eq:enident}. 
We similarly infer the bound 
\begin{align*}
&\left| \int_T^{T_1}\int_{\R^5}r^{-4}F(vr)\cos(2u_{approx})\cdot 2u_{approx, t}\,dxdt\right|\\
&\lesssim \frac{C\delta_1}{\gamma}(C_1\delta_1)^2\left(\frac{T_1}{T}\right)^{2\gamma},
\end{align*}
which suffices provided $\delta_1\ll \gamma$. 
\\

For the fourth term on the right hand side of \eqref{eq:enident}, we have to argue slightly differently, since now the smallness has to come from $u_{approx}$, which however is large immediately outside the light cone. Write 
\begin{align*}
&\int_T^{T_1}\int_{\R^5}\frac{2v}{r^2}\big[\cos(2u_{approx}) - 1\big] v_t\,dx dt\\
& = \int_{\R^5}\frac{v^2}{r^2}\big[\cos(2u_{approx}) - 1\big]\,dx|_{T}^{T_1}\\
& + \int_T^{T_1}\int_{\R^5}\frac{v^2}{r^2}\sin(2u_{approx})u_{approx, t}\,dx dt.\\
\end{align*}
For the first term on the right and evaluated at $t = T_1$, we get 
\begin{align*}
 \int_{\R^5}\frac{v^2}{r^2}\big[\cos(2u_{approx}) - 1\big]\,dx|_{T_1} &=  \int_{r< T_1 - C}\frac{v^2}{r^2}\big[\cos(2u_{approx}) - 1\big]\,dx|_{T_1}\\
 & + \int_{r\in[T_1 - C, T_1+C]}\frac{v^2}{r^2}\big[\cos(2u_{approx}) - 1\big]\,dx|_{T_1}.\\
\end{align*}
For the first term use that $|\cos(2u_{approx}) - 1|\ll_{\tilde{d}_2} 1$, whence by Hardy's inequality 
\begin{align*}
\left|\int_{r< T_1 - C}\frac{v^2}{r^2}\left[\cos(2u_{approx}) - 1\right]\,dx|_{T_1}\right|\ll \big\|\nabla_x v(T_1,\cdot)\big\|_{L_x^2}^2,
\end{align*}
which can be absorbed on the right hand side of \eqref{eq:enident}. 
For the remainder term, smallness has to be a consequence of the additional $r$-localization. In fact, from Strauss' inequality for radial functions, we infer 
\[
\left|v(t, r)\right|\lesssim r^{-\frac{3}{2}}\left\|v(t,\cdot)\right\|_{\dot{H}^1},
\]
and so 
\begin{align*}
&\left|\int_{r\in[T_1 - C, T_1+C]}\frac{v^2}{r^2}\big[\cos(2u_{approx}) - 1\big]\,dx|_{T_1}\right|\\
&\lesssim \big\|v(T_1,\cdot)\big\|_{\dot{H}^1}^2\int_{r\in[T_1 - C, T_1+C]}r^{-5}\cdot r^4\,dr\ll  \big\|v(T_1,\cdot)\big\|_{\dot{H}^1}^2
\end{align*}
since $T_1>T\gg 1$ by assumption. Hence this term can be absorbed on the right hand side of \eqref{eq:enident}. 
For the space time integral above, we similarly divide it into an integral over $r<t-C, r\in [t-C, t+C]$, and by similar reasoning we obtain 
\begin{align*}
&\left| \int_T^{T_1}\int_{\R^5}\frac{v^2}{r^2}\sin(2u_{approx})u_{approx, t}\,dx dt\right|\\
&\ll \int_T^{T_1}\frac{\big\|\nabla_x(t, \cdot)\big\|_{L_x^2}^2}{t}\,dt
\end{align*}
where the implied constant depends on $\tilde{d}_1$ as well as $T$(in particular, the latter needs to be large enough in relation to $\tilde{d}_1$ for this term to be small), and so we can again close provided the implied constant is small enough in relation to $\gamma$. 
\\

Finally, to control the last term on the right in \eqref{eq:enident}, we use 
\begin{align*}
\left|\int_T^{T_1}\int_{\R^5}v_t\frac{e_0}{r}\,dxdt\right|&\lesssim \int_T^{T_1}\|v_t(t, \cdot)\|_{L_x^2}\|\frac{e_0}{r}(t,\cdot)\|_{L_x^2}\,dt\\
&\lesssim \int_T^{T_1}\left(\frac{\|v_t(t, \cdot)\|_{L_x^2}^2}{t^2} + t^{-2}\right)\,dt\ll (C_1\delta_1)^2\left(\frac{T_1}{T}\right)^{2\gamma}
\end{align*}
provided we pick $T$ sufficiently large. This completes the bootstrap for the norm 
\[
\sup_{t\in [T, T_1]}\big\|v[t, \cdot]\big\|_{(\frac{t}{T})^{\gamma}\dot{H}^1(\R^5)\times (\frac{t}{T})^{\gamma}L^2(\R^5)}.
\]

We continue with the Strichartz type norms of critical scaling, given by 
\[
\big\|v\big\|_{L_t^2(L_x^{10}\cap \nabla_x^{-\frac{1}{2}}L_x^5)([T, T_1]\times \R^5)} + \big\|(v, v_t)\big\|_{L_t^\infty \dot{H}^{\frac{3}{2}}\times L_t^\infty \dot{H}^{\frac{1}{2}}([T, T_1]\times \R^5)}.
\]
Using the standard Strichartz estimates for free waves on $\R^{5+1}$, it suffices to prove the bound
\[
\big\|F(v)\big\|_{L_t^1\dot{H}^{\frac{1}{2}}([T, T_1]\times \R^5)}\ll C_1\delta_1
\]
where $F(v)$ denotes the right hand side of \eqref{eq:veqn}. We estimate the individual components on the right: 
\\

{\it{The contribution of $ -\frac{1}{r}\frac{\sin(2\epsilon) - 2\epsilon}{r^2}\cos(2u_{approx})$.}} We can bound this by 
\begin{align*}
&\left\|\frac{1}{r}\frac{\sin(2\epsilon) - 2\epsilon}{r^2}\cos(2u_{approx})\right\|_{L_t^1 \dot{H}^{\frac{1}{2}}([T, T_1]\times \R^5)}\\
&\lesssim \big\| v^3\big\|_{L_t^1 \dot{H}^{\frac{1}{2}}([T, T_1]\times \R^5)} + \big\|v^3(\nabla_x^{\frac{1}{2}}v) r\big\|_{L_t^1L_x^2([T, T_1]\times \R^5)}
+\big\| v^4 r^{\frac{1}{2}}\big\|_{L_t^1L_x^2([T, T_1]\times \R^5)}\\
& + \big\|v^3 t^{-\frac{1}{2}}\big\|_{L_t^1L_x^2([T, T_1]\times \R^5)},
\end{align*}
where we have taken advantage of writing 
\[
\frac{1}{r}\frac{\sin(2\epsilon) - 2\epsilon}{r^2}\cos(2u_{approx}) = v^3 \frac{\sin(2\epsilon) - 2\epsilon}{\epsilon^3}\cos(2u_{approx})
\]
and also used the fractional derivative Leibniz rule. 
Then we estimate 
\begin{align*}
\|v^3\|_{L_t^1\dot{H}^{\frac{1}{2}}([T, T_1]\times \R^5)}&\lesssim \|\nabla_x^{\frac{1}{2}}v\|_{L_t^\infty L_x^{\frac{10}{3}}([T, T_1]\times \R^5)}\|v\|_{L_t^2 L_x^{10}([T, T_1]\times \R^5)}^2\\&\lesssim \|v\|_{L_t^\infty \dot{H}^{\frac{3}{2}}([T, T_1]\times \R^5)}\|v\|_{L_t^2 L_x^{10}([T, T_1]\times \R^5)}^2\\
&\lesssim (C_1\delta_1)^3\ll C_1\delta_1.
\end{align*}
Next, taking advantage of the Strauss' inequality $|v(t,r)|\lesssim r^{-1}\big\|v(t, \cdot)\big\|_{\dot{H}^{\frac{3}{2}}}$, we have 
\begin{align*}
&\big\|v^3(\nabla_x^{\frac{1}{2}}v) r\big\|_{L_t^1L_x^2([T, T_1]\times \R^5)}
+\big\| v^4 r^{\frac{1}{2}}\big\|_{L_t^1L_x^2([T, T_1]\times \R^5)}\\
&\lesssim C_1\delta_1\big\|v^2(\nabla_x^{\frac{1}{2}}v)\big\|_{L_t^1L_x^2([T, T_1]\times \R^5)} + (C_1\delta_1)^{\frac{1}{2}}\big\|v^{\frac{7}{2}}\big\|_{L_t^1L_x^2([T, T_1]\times \R^5)}\\
&\lesssim  C_1\delta_1\big\|v\big\|_{L_t^2 L_x^{10}([T, T_1]\times \R^5)}^2\big\|\nabla_x^{\frac{1}{2}}v\big\|_{L_t^\infty L_x^{\frac{10}{3}}([T, T_1]\times \R^5)}\\
&+(C_1\delta_1)^{\frac{1}{2}}\big\|v\big\|_{L_t^2 L_x^{10}([T, T_1]\times \R^5)}^2\big\|v^{\frac{3}{2}}\big\|_{L_t^\infty L_x^{\frac{10}{3}}([T, T_1]\times \R^5)}\\
&\lesssim (C_1\delta_1)^{\frac{1}{2}}\big\|v\big\|_{L_t^2 L_x^{10}([T, T_1]\times \R^5)}^2\big\|v\big\|_{L_t^\infty \dot{H}^{\frac{3}{2}}([T, T_1]\times \R^5)}\ll C_1\delta_1.\\
\end{align*}
Finally, for the contribution of $v^3t^{-\frac{1}{2}}$, we use the fact that by the Huyghen's principle, the support of the function $v(t, r)$ is contained in the set $r<t+C$, and so 
\begin{align*}
&\big\|v^3 t^{-\frac{1}{2}}\big\|_{L_t^1L_x^2([T, T_1]\times \R^5)}\lesssim \big\|v\big\|_{L_t^2 L_x^{10}([T, T_1]\times \R^5)}^2\big\|v\big\|_{L_t^\infty L_x^{5}([T, T_1]\times \R^5)}\big\|t^{-\frac{1}{2}}\big\|_{L_t^\infty L_x^{10}([T, T_1]\times \{r\leq t+C\})}\\
&\lesssim (C_1\delta_1)^3. 
\end{align*}

{\it{The contribution of $-\frac{2\epsilon}{r^3}\big[\cos(2u_{approx}) - 1\big]$.}}  Here we distinguish between high and low frequency factors. Specifically, we write schematically
\[
\frac{2\epsilon}{r^3}\left[\cos(2u_{approx}) - 1\right] = \frac{2\epsilon}{r}P_{<t^{-\delta}}\left[\frac{\cos(2u_{approx}) - 1}{r^2}\right] + \frac{2\epsilon}{r}P_{\geq t^{-\delta}}\left[\frac{\cos(2u_{approx}) - 1}{r^2}\right]. 
\]
For the second term on the right, we exploit that 
\[
P_{\geq t^{-\delta}}\left[\frac{\cos(2u_{approx}) - 1}{r^2}\right] 
\]
enjoys a special smallness property. In fact, by direct computation, we get 
\begin{align*}
\nabla_x\left[\frac{\cos(2u_{approx}) - 1}{r^2}\right]  &= \nabla_x\left[\frac{u_{approx}^2}{r^2}\frac{\cos(2u_{approx}) - 1}{u_{approx}^2}\right]\\
&= \left[\nabla_x(\frac{u_{approx}}{r})\frac{u_{approx}}{r}\frac{\cos(2u_{approx}) - 1}{u_{approx}^2}\right]\\
&+\left[\frac{u_{approx}^2}{r^2}\nabla_x\left(\frac{\cos(2u_{approx}) - 1}{u_{approx}^2}\right)\right]\\
\end{align*}
and we bound these terms by $O\left(\frac{\log t}{t^3}\right)$. It follows that 
\[
\left|P_{\geq t^{-\delta}}\left[\frac{\cos(2u_{approx}) - 1}{r^2}\right]\right|\lesssim \frac{\log t}{t^{3-\delta}}.
\]
This allows to bound the high frequency term by 
\begin{align*}
\left\|\frac{2\epsilon}{r}P_{\geq t^{-\delta}}\left[\frac{\cos(2u_{approx}) - 1}{r^2}\right]\right\|_{\dot{H}^{\frac{1}{2}}(\R^5)}&\lesssim \left\|v\left\|_{L_x^{10}}\right\|\chi_{r\lesssim t}\right\|_{L_x^5}\left\|P_{\geq t^{-\delta}}\nabla_x^{\frac{1}{2}}\left[\frac{\cos(2u_{approx}) - 1}{r^2}\right]\right\|_{L_x^5(r\lesssim t)}\\
& + \big\|\nabla_x^{\frac{1}{2}}v\big\|_{L_x^{5}}\big\|\chi_{r\lesssim t}\big\|_{L_x^5}\left\|P_{\geq t^{-\delta}}\left[\frac{\cos(2u_{approx}) - 1}{r^2}\right]\right\|_{L_x^{10}(r\lesssim t)}.\\
\end{align*}
We conclude that 
\begin{align*}
&\left\|\frac{2\epsilon}{r}P_{\geq t^{-\delta}}\left[\frac{\cos(2u_{approx}) - 1}{r^2}\right]\right\|_{L_t^1\dot{H}^{\frac{1}{2}}([T, T_1]\times \R^5)}\\
&\lesssim \big\|v\big\|_{L_t^2 L_x^{10}([T, T_1]\times \R^5)}\left\|\left\| tP_{\geq t^{-\delta}}\nabla_x^{\frac{1}{2}}\left[\frac{\cos(2u_{approx}) - 1}{r^2}\right]\right\|_{L_x^{5}(r\lesssim t)}\right\|_{L_t^2[T, T_1]}\\
& +  \left\|\nabla_x^{\frac{1}{2}}v\right\|_{L_t^2L_x^{5}([T, T_1]\times \R^5)}\left\|\left\| tP_{\geq t^{-\delta}}\left[\frac{\cos(2u_{approx}) - 1}{r^2}\right]\right\|_{L_x^{10}(r\lesssim t)}\right\|_{L_t^2[T, T_1]}\\
&\lesssim \left[ \big\|v\big\|_{L_t^2 L_x^{10}([T, T_1]\times \R^5)} +  \big\|\nabla_x^{\frac{1}{2}}v\big\|_{L_t^2L_x^{5}([T, T_1]\times \R^5)}\right]\left\|\frac{\log t}{t^{1-\delta}}\right\|_{L_t^2[T, T_1]}\\
&\ll C_1\delta_1
\end{align*}
on account of $T\gg 1$. 
\\
Next, consider the low frequency term 
\begin{align*}
&\frac{2\epsilon}{r}P_{<t^{-\delta}}\left[\frac{\cos(2u_{approx}) - 1}{r^2}\right]\\
&=P_{<t^{-\frac{\delta}{2}}}\left(\frac{2\epsilon}{r}\right)P_{<t^{-\delta}}\left[\frac{\cos(2u_{approx}) - 1}{r^2}\right]\\
&+P_{\geq t^{-\frac{\delta}{2}}}\left(\frac{2\epsilon}{r}\right)P_{<t^{-\delta}}\left[\frac{\cos(2u_{approx}) - 1}{r^2}\right].\\
 \end{align*}
For the second term on the right, we get 
\begin{align*}
&\left\|P_{\geq t^{-\frac{\delta}{2}}}\left(\frac{2\epsilon}{r}\right)P_{<t^{-\delta}}\left[\frac{\cos(2u_{approx}) - 1}{r^2}\right]\right\|_{L_t^1\dot{H}^{\frac{1}{2}}([T, T_1]\times \R^5)}\\
&\lesssim \left\| t^{-\frac{3}{4}\delta}\langle\nabla_x\rangle^{\frac{1}{2}}P_{\geq t^{-\frac{\delta}{2}}}v\right\|_{L_t^\infty L_x^2([T, T_1]\times \R^5)}\left\|t^{\frac{3}{4}\delta}\langle\nabla_x\rangle^{\frac{1}{2}}P_{<t^{-\delta}}\left[\frac{\cos(2u_{approx}) - 1}{r^2}\right]\right\|_{L_t^1 L_x^\infty}
\end{align*}
and then use that from our bootstrap hypothesis we have 
\[
 \left\| t^{-\frac{3}{4}\delta}\langle\nabla_x\rangle^{\frac{1}{2}}P_{\geq t^{-\frac{\delta}{2}}}v\right\|_{L_t^\infty L_x^2([T, T_1]\times \R^5)}\lesssim C_1\delta_1
 \]
 while using Bernstein's inequality, we have 
 \begin{align*}
&\left\|t^{\frac{3}{4}\delta}\langle\nabla_x\rangle^{\frac{1}{2}}P_{<t^{-\delta}}\left[\frac{\cos(2u_{approx}) - 1}{r^2}\right]\right\|_{L_t^1 L_x^\infty([T, T_1]\times \R^5)}\\
&\lesssim \left\|\frac{u_{approx}}{r}\right\|_{L_t^2 L_x^{10+}([T, T_1]\times \R^5)}^2\ll 1
\end{align*}
on account of $T\gg 1$. The conclusion is that 
\begin{align*}
\left\|P_{\geq t^{-\frac{\delta}{2}}}\left(\frac{2\epsilon}{r}\right)P_{<t^{-\delta}}\left[\frac{\cos(2u_{approx}) - 1}{r^2}\right]\right\|_{L_t^1\dot{H}^{\frac{1}{2}}([T, T_1]\times \R^5)}\ll C_1\delta_1. 
\end{align*}
On the other hand, for the term where all factors have low frequency, i. e. 
\[
P_{<t^{-\frac{\delta}{2}}}\left(\frac{2\epsilon}{r}\right)P_{<t^{-\delta}}\left[\frac{\cos(2u_{approx}) - 1}{r^2}\right],
\]
we exploit the extra outer derivative and low frequency control (with a small loss): we have at fixed time $t$
\begin{align*}
&\left\|P_{<t^{-\frac{\delta}{2}}}(\frac{2\epsilon}{r})P_{<t^{-\delta}}\left[\frac{\cos(2u_{approx}) - 1}{r^2}\right]\right\|_{\dot{H}^{\frac{1}{2}}(\R^5)}\\
&\lesssim \left\|\nabla_x^{\frac{1}{2}}P_{<t^{-\frac{\delta}{2}}}v\right\|_{L_x^{\frac{10}{3}}}\left\|P_{<t^{-\delta}}\left[\frac{\cos(2u_{approx}) - 1}{r^2}\right]\right\|_{L_x^5(r\lesssim t)}\\&
+ \left\|P_{<t^{-\frac{\delta}{2}}}v\right\|_{L_x^{\frac{10}{3}}}\left\|\nabla_x^{\frac{1}{2}}P_{<t^{-\delta}}\left[\frac{\cos(2u_{approx}) - 1}{r^2}\right]\right\|_{L_x^5(r\lesssim t)}
\lesssim t^{-\frac{\delta}{4}+\gamma - 1}T^{-\gamma}C_1\delta_1
\end{align*}
where we have taken advantage of 
\[
 \left\|P_{<t^{-\frac{\delta}{2}}}v(t, \cdot)\right\|_{L_x^{\frac{10}{3}}}\lesssim \left\|v(t, \cdot)\right\|_{\dot{H}^1}\lesssim \left(\frac{t}{T}\right)^{\gamma}.
 \]
 Also, the additional factors $t^{-\frac{\delta}{4}}$ which ensure integrability stem from the operator $\nabla_x^{\frac{1}{2}}P_{<t^{-\delta}}$. 
Thus if we arrange (as we may) that $\gamma\ll \delta$, we find 
 \begin{align*}
 &\left\|P_{<t^{-\frac{\delta}{2}}}(\frac{2\epsilon}{r})P_{<t^{-\delta}}\left[\frac{\cos(2u_{approx}) - 1}{r^2}\right]\right\|_{L_t^1\dot{H}^{\frac{1}{2}}([T, T_1]\times\R^5)}&\lesssim C_1\delta_1\left\|t^{-\frac{\delta}{4}+\gamma - 1}T^{-\gamma}\right\|_{L_t^1[T, T_1]}\\
 &\ll C_1\delta_1
 \end{align*}
provided $T$ is sufficiently large (in relation to $\delta^{-1}$). This finally concludes bounding the contribution from 
\[
-\frac{2\epsilon}{r^3}\left[\cos(2u_{approx}) - 1\right]. 
\]

{\it{The contribution of $\frac{\sin(2u_{approx})}{r^3}\big(\cos(2\epsilon) - 1\big)$.}} We again use the high-low frequency method, which is somewhat simpler to implement here: write 
\begin{align*}
&\frac{\sin(2u_{approx})}{r^3}\big(\cos(2\epsilon) - 1\big)\\
&=P_{<t^{-\delta}}\left(\frac{\sin(2u_{approx})}{r}\right)\left(\frac{\cos(2\epsilon) - 1}{r^2}\right) + P_{\geq t^{-\delta}}\left(\frac{\sin(2u_{approx})}{r}\right)\left(\frac{\cos(2\epsilon) - 1}{r^2}\right).
\end{align*}
For the second term on the right, use that 
\begin{align*}
&\left\|\langle\nabla_x\rangle^{\frac{1}{2}}P_{\geq t^{-\delta}}\left(\frac{\sin(2u_{approx})}{r}\right)\right\|_{L_x^{\frac{10}{3}}(r\lesssim t)}\lesssim \frac{\log t}{t^{\frac{1}{2}-\delta}},\\
&\left\|P_{\geq t^{-\delta}}\left(\frac{\sin(2u_{approx})}{r}\right)\right\|_{L_x^{10}(r\lesssim t)}\lesssim \frac{\log t}{t^{\frac{3}{2}-\delta}},
\end{align*}
and so we find 
\begin{align*}
&\left\|P_{\geq t^{-\delta}}\left(\frac{\sin(2u_{approx})}{r}\right)\left(\frac{\cos(2\epsilon) - 1}{r^2}\right)\right\|_{L_t^1\dot{H}^{\frac{1}{2}}([T, T_1]\times\R^5)}\\
&\lesssim \left\|\nabla_x^{\frac{1}{2}}P_{\geq t^{-\delta}}\left(\frac{\sin(2u_{approx})}{r}\right)\right\|_{L_t^\infty L_x^{\frac{10}{3}}([T, T_1]\times\R^5)}\left\|v\right\|_{L_t^2L_x^{10}([T, T_1]\times\R^5)}^2\\
&+ \left\|P_{\geq t^{-\delta}}\left(\frac{\sin(2u_{approx})}{r}\right)\right\|_{L_t^2 L_x^{10}([T, T_1]\times\R^5)}\left\|\nabla^{\frac{1}{2}}v\right\|_{L_t^\infty L_x^{\frac{10}{3}}([T, T_1]\times\R^5)}\left\|v\right\|_{L_t^2 L_x^{10}([T, T_1]\times\R^5)}\\
&\lesssim (C_1\delta_1)^2\ll C_1\delta_1. 
\end{align*}
For the low frequency term, we get 
\begin{align*}
&\left\|P_{<t^{-\delta}}\left(\frac{\sin(2u_{approx})}{r}\right)\left(\frac{\cos(2\epsilon) - 1}{r^2}\right)\right\|_{L_t^1\dot{H}^{\frac{1}{2}}([T, T_1]\times\R^5)}\\
&\lesssim \left\|\nabla_x^{\frac{1}{2}}v\right\|_{L_t^2 L_x^5([T, T_1]\times\R^5)}\left\|t^{-\gamma}v\right\|_{L_t^\infty L_x^{\frac{10}{3}}([T, T_1]\times\R^5)}\left\| t^{\gamma}P_{<t^{-\delta}}\left(\frac{\sin(2u_{approx})}{r}\right)\right\|_{L_t^2 L_x^{\infty}}\\
& + \left\|v\right\|_{L_t^2 L_x^{10}([T, T_1]\times\R^5)}\left\|t^{-\gamma}v\right\|_{L_t^\infty L_x^{\frac{10}{3}}([T, T_1]\times\R^5)}\left\| t^{\gamma}\nabla_x^{\frac{1}{2}}P_{<t^{-\delta}}\left(\frac{\sin(2u_{approx})}{r}\right)\right\|_{L_t^2 L_x^{10}}\\
&\lesssim (C_1\delta_1)^2\ll C_1\delta_1. 
\end{align*}

{\it{The contribution of $\frac{e_0}{r}$.}} Here we immediately check that $\left\|\frac{e_0}{r}\right\|_{L_t^1\dot{H}^{\frac{1}{2}}([T, T_1]\times\R^5)}\lesssim T^{-1}\ll C_1\delta_1$ if $T$ is sufficiently large,  which is as desired. 
The proof of Proposition~\ref{prop:bootstrap} is thereby concluded.

\section{Proof of the main result}
Here we shall show how to conclude Theorem \ref{th:main} building on the previous sections.
Indeed Theorem \ref{th:main} follows from Proposition \ref{prop:vMain}, the solution to \eqref{eq:scalar wave equation}
being given by $u:=u_{approx}+ \epsilon(t, r)$ where $\epsilon(t, r):= r v$ with $v$ provided by Proposition \ref{prop:vMain}.
The initial data $(f,g)$ are given by $(u(T, \cdot), u_t(T, \cdot))$ where $T>0$ depending on $\tilde{d}_1$ is given by Proposition \ref{prop:vMain}.
Since $(v, v_t)$ does have finite critical norm, but the approximate solution $(u_{approx}, \partial_t u_{approx})$ does not, we easily conclude that 
the initial data have infinite critical norm.
Clearly the perturbation $(\epsilon, \epsilon_t)$ lies in the spaces $\dot{H}^{s}\times \dot{H}^{s-1}$ for $s>\frac{3}{2}$ by construction (remind that $(v,v_t)$ is compactly supported).
Moreover, due to the asymptotics for $r\rightarrow \infty$ of the self--similar solutions given by formulas \eqref{eq:first behavior at infty} and \eqref{eq:behavior at infty}, in the small and large case respectively,
we have that $( u_{approx}- q_1, \partial_t u_{approx})$, $( u_{approx}- c_1, \partial_t u_{approx})$ respectively, has finite norm in $\dot{H}^{s}\times \dot{H}^{s-1}$ for $s>\frac{3}{2}$: that is how we understand the finiteness in 
$\dot{H}^{s}\times \dot{H}^{s-1}$ for $s>\frac{3}{2}$ of the data $(f,g)$ as claimed in Theorem \ref{th:main}.
Of course the condition $\| f\|_{L^\infty(r\geq1)}> M,$ for arbitrary $M>0$, can be achieved simply by choosing $\wt{d}_1> M^2$ in the context of large self--similar solutions as
provided by Lemma \ref{l:lemma5}.
Finally the stability under a certain class of perturbations is a consequence of the fact that $v$ belongs to an open set with respect to the norms of Proposition \ref{prop:local exist}.

\end{proof}

\end{document}